%% file: BOS.tex
\documentclass%[12pt]
{amsproc}

\usepackage{amsmath}
\usepackage{amsthm}
\usepackage{amssymb}
\usepackage{graphicx}
\usepackage{multirow}
\usepackage{tikz-cd}
 
\DeclareGraphicsExtensions{.png,.pdf,.eps}
\usepackage[all,2cell]{xy}
\usepackage{tikz}
\usepackage{pgf}
\usepackage{enumerate}
\usepackage{mathdots}
\usetikzlibrary{cd}
\usepackage{caption}
\usepackage{adjustbox}
\usetikzlibrary{matrix,arrows,backgrounds}
\usepackage{circuitikz}

\usepackage{listings}
\usepackage{pythonhighlight}

\usepackage{enumitem}

\usepackage{array}
\newcolumntype{H}{>{\setbox0=\hbox\bgroup}c<{\egroup}@{}}
\usepackage{hyperref}

\newcommand\xx{11} %% MARGINS OF ITEMIZE
\newcommand\yy{18} %% MARGINS OF ENUMERATE
\newtheorem{theorem}{Theorem}[section]
\newtheorem{lemma}[theorem]{Lemma}

\newtheorem{proposition}[theorem]{Proposition}

\newtheorem{construction}{Construction}

\theoremstyle{definition}

\newtheorem{definition}[theorem]{Definition}

\newtheorem{setup}[theorem]{Setup}

\newtheorem{remark}[theorem]{Remark}

\newtheorem{set}[theorem]{Setup}

\title[Constructing geometric realizations of birational maps 
]{Constructing geometric realizations of birational maps between Mori Dream Spaces}

\author[Barban]{Lorenzo Barban}
\address{Center for Complex Geometry, Institute for Basic Science (IBS), 55 Expo-ro, Yuseong-gu, Daejeon, 34126, Republic of Korea}
\email{lorenzobarban@ibs.re.kr}

\author[Occhetta]{Gianluca Occhetta}
\address{Dipartimento di Matematica, Universit\`a degli Studi di Trento, via
Sommarive 14 I-38123, Trento (TN), Italy}
\email{gianluca.occhetta@unitn.it}

\author[Sol\'a Conde]{Luis E. Sol\'a Conde}
\address{Dipartimento di Matematica, Universit\`a degli Studi di Trento, via Sommarive 14 I-38123, Trento (TN), Italy}
\email{eduardo.solaconde@unitn.it}


\input{definitions.tex}

\begin{document}
\begin{abstract}
We construct geometric realizations --projective algebraic versions of cobordisms-- for birational maps between Mori Dream Spaces. We show that these geometric realizations are Mori Dream Spaces, as well, and that they can be constructed so that they induce factorizations of the original birational maps as compositions of wall-crossings. 
In the case of toric birational maps between normal $\Q$-factorial, projective toric varieties, we provide several {\tt SageMath} functions to work with $\C^*$-actions and birational geometry; in particular we show how to explicitly construct a moment polytope of a toric geometric realization. Moreover, by embedding Mori Dream Spaces in toric varieties, we obtain geometric realizations of birational maps of Mori Dream Spaces as restrictions of toric geometric realizations. We also provide examples and discuss when a geometric realization is Fano. 
\end{abstract}
\maketitle

%%%%%%%%%%%%%%%%%%%%%%%%%%%%%%%%%%%%%%%%%%%%%%%%%%%%%%%%%%%%%%%%%%%%%%%%%%%%%
%%%%%%%%%%%%%%%%%%%%%%%%%%%%%%%%%%%%%%%%%%%%%%%%%%%%%%%%%%%%%%%%%%%%%%%%%%%%%
%%
%% SECTION 1 
%%
%%%%%%%%%%%%%%%%%%%%%%%%%%%%%%%%%%%%%%%%%%%%%%%%%%%%%%%%%%%%%%%%%%%%%%%%%%%%%
%%%%%%%%%%%%%%%%%%%%%%%%%%%%%%%%%%%%%%%%%%%%%%%%%%%%%%%%%%%%%%%%%%%%%%%%%%%%%
\section{Introduction}

Bia{\l}ynicki-Birula theory for $\C^*$-actions serves as an algebraic analog of Morse theory (see \cite{Carrell,BB}), that has deep connections with birational geometry, as shown in the works of Reid, Morelli, W{\l}odarczyk, and others (cf. \cite{ReidFlip, Morelli, Wlodarczyk}). Within this circle of ideas, it has been recently shown that $\C^*$-actions on projective varieties can be encoded through birational transformations on geometric quotients, defined by the properties of the action on its fixed point components (see \cite{WORS4,BRUS} and the references therein).

Consider a birational map $\phi:Y_-\dashrightarrow Y_+$ between  normal projective varieties. A geometric realization of $\phi$ is a normal projective variety $X$, endowed with a faithful $\C^*$-action, where $Y_\pm$ correspond to geometric quotients of $X$ with respect to different linearizations of the action on an ample line bundle $L$, such that, for general $y_\pm \in Y_\pm$, $y_+ = \phi(y_-)$ if and only if they represent the same orbit in $X$. Typically, we assume that $Y_\pm$ and $X$ have reasonable singularities, consistent with the framework of modern birational geometry.

The geometric quotients $\GX_i$ of $X$ associated with intermediate weights of the $\C^*$-action on $L$ form a factorization of $\phi$:
\[
\xymatrix{Y_-\ar@{=}[r]\ar@<1ex>@/^1.3pc/@{-->}[rrrrrr]^{\phi}&\GX_1\ar@{-->}_{\phi_1}[r]&\GX_2\ar@{-->}_{\phi_2}[r]&\ldots \ar@{-->}_{\phi_{s-3}}[r]&\GX_{s-2}\ar@{-->}_{\phi_{s-1}}[r]&\GX_s\ar@{=}[r]&Y_+.}
\]

Explicit examples of geometric realizations have been constructed for classical birational maps, including quadro-quadric Cremona transformations, inversions of (projectivizations of) some Jordan algebras, and Atiyah flips, among others (see \cite{WORS1,FSC,BF}).
A particularly enlightening example is the following; consider the cubo--cubic Cremona transformation of $\P^3$:
\[
\phi: \P^3\dashrightarrow \P^3,\qquad (x_0:x_1:x_2:x_3)\mapsto (x_0^{-1}:x_1^{-1}:x_2^{-1}:x_3^{-1})
\]
A geometric realization of $\phi$ is given by the blow up of $X=\P^1\times\P^1\times\P^1\times\P^1$ along two extremal fixed points, endowed with the diagonal $\C^*$-action. Linearizing the action on the ample line bundle $\cO_X(1,1,1,1)$, we get four geometric quotients, and a factorization of $\phi$ in three birational maps
\[
\xymatrix{\P^3\ar@{=}[r]\ar@<1ex>@/^1pc/@{-->}[rrrrr]^{\phi}&\GX_1\ar@{-->}_{\phi_1}[r]&\GX_2\ar@{-->}_{\phi_2}[r]&\GX_{3}\ar@{-->}_{\phi_{3}}[r]&\GX_4\ar@{=}[r]&\P^3.}
\]
that can be described as follows:
\begin{itemize}[leftmargin =\xx pt]
	\item $\phi_1$ is a divisorial extraction (inverse of the blowup of $\P^3$ along four points);
	\item $\phi_2$ is the composition of the flips of $6$ lines;
	\item $\phi_3$ is a divisorial contraction (blowup of  $\P^3$ along four points).
\end{itemize}
\begin{figure}[htp]
	\centering
	\hfill
	\includegraphics[scale=0.06]{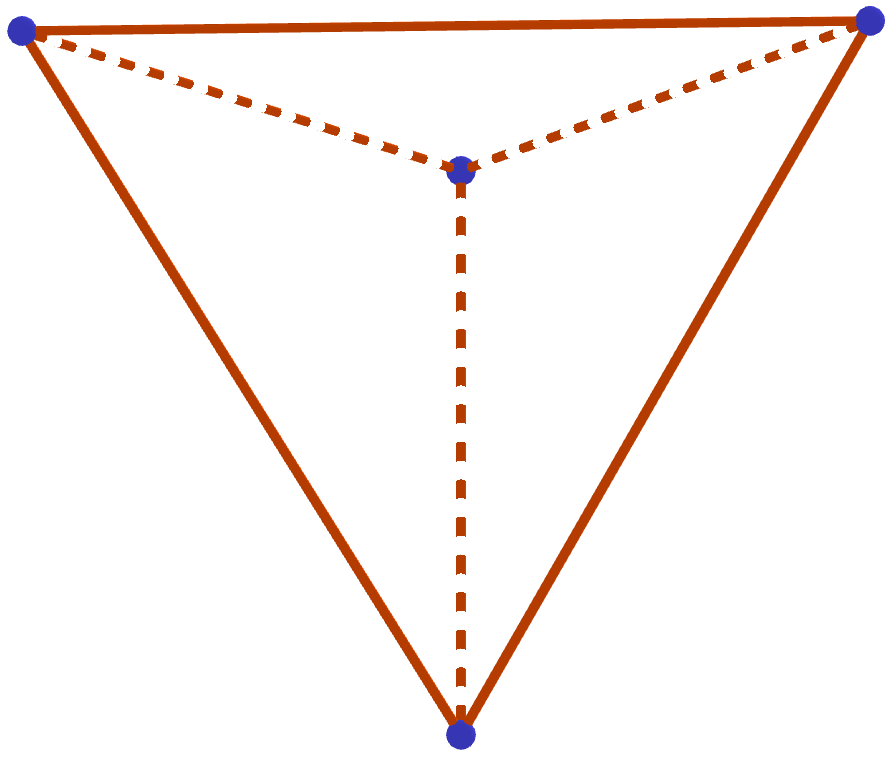}\hfill 
	\includegraphics[scale=0.06]{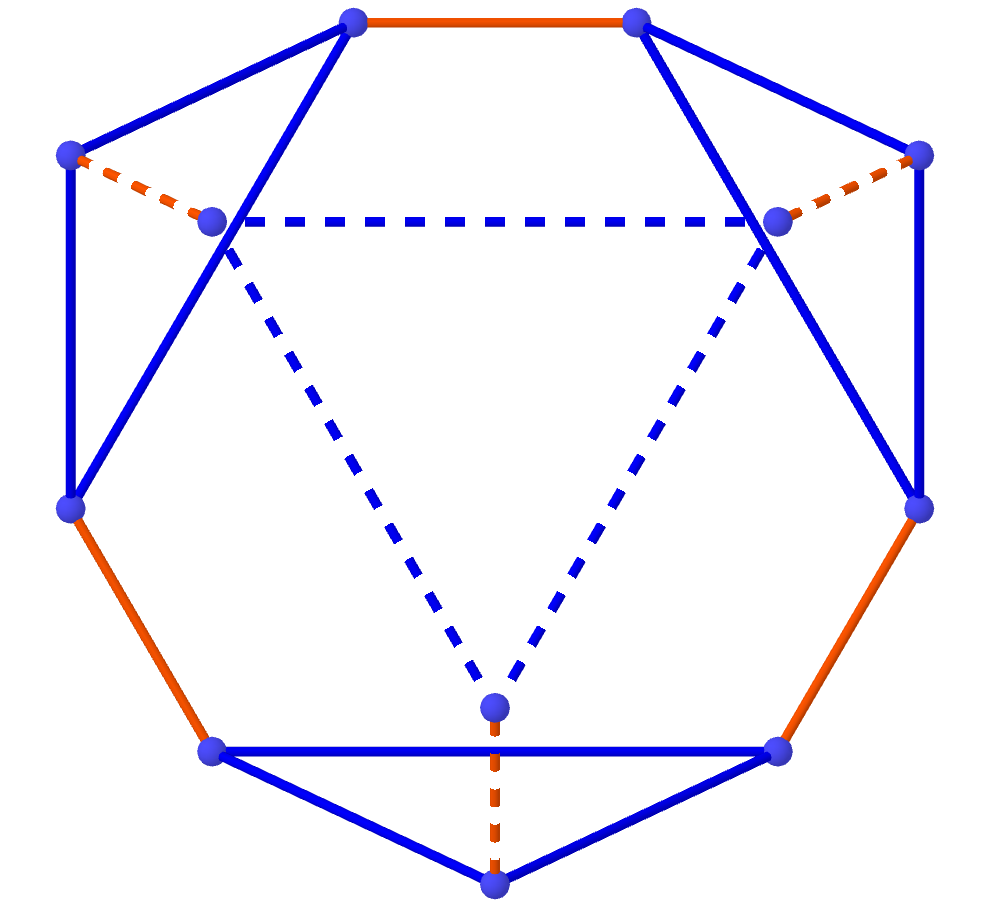}\hfill 
	\includegraphics[scale=0.06]{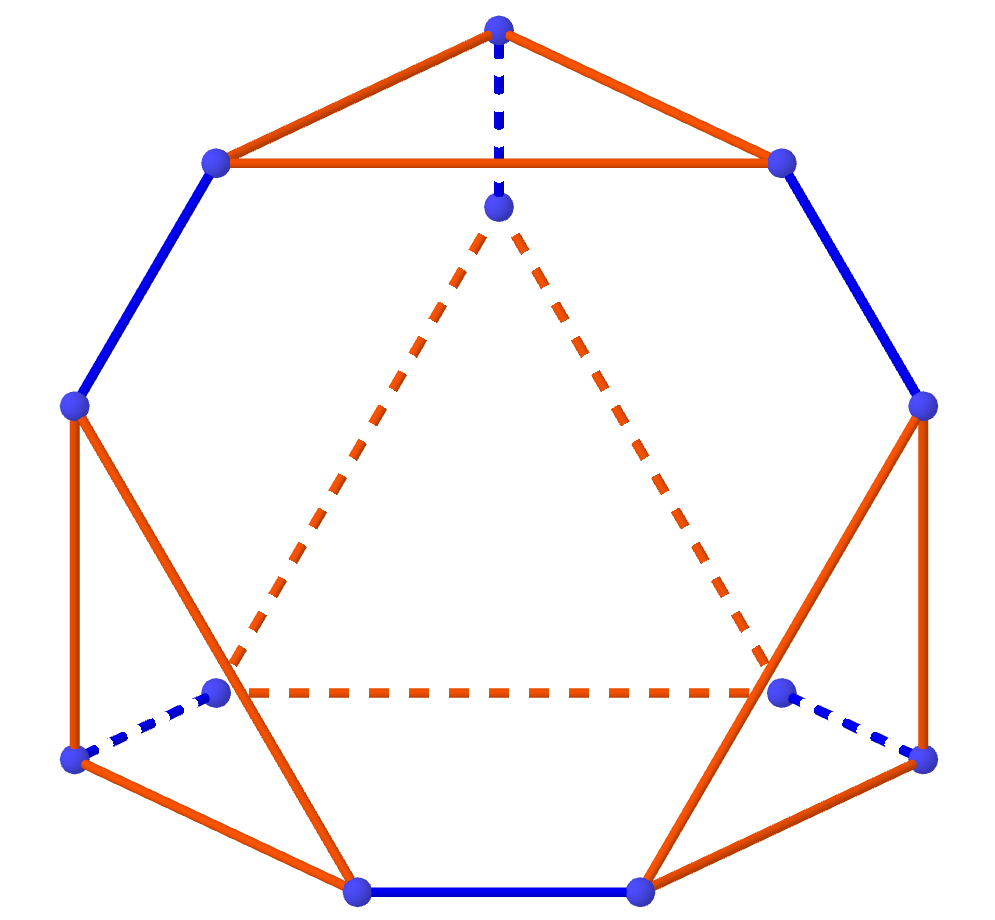}\hfill 
	\includegraphics[scale=0.06]{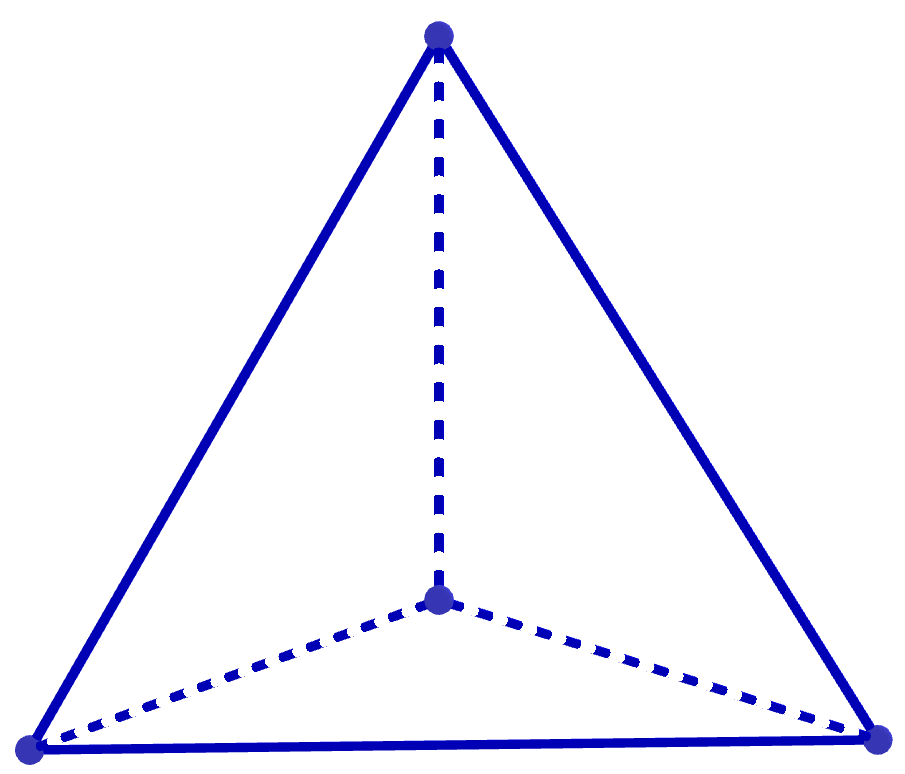}\hfill 
	\hfill
	\caption{\label{fig:cremonafact}Moment polytopes of the varieties $\P^3$, $\GX_2,\GX_3,\P^3$.}
\end{figure}
Note that the maps $\phi_i$ in this example are not elementary; however,
by consider a linearization on an appropriately chosen ample line bundle $L=\cO_X(a,b,c,d)$, it is always possible to obtain a factorization into elementary divisorial extractions, contractions and flips (see \cite[Section~6]{BRUS} for details).

We call a geometric realization {\em sharp} if the maps $\phi_i$ in the factorization are elementary flips, divisorial contractions, or divisorial extractions.  This concept links the Sarkisov program for $\phi$ to $\C^*$-equivariant minimal model programs (MMPs) for its geometric realization $X$.

In this paper we investigate the existence of geometric realizations, and which geometrical property of $Y_{\pm}$ are inherited by the geometric realization. Our main statement is the following:

\begin{theorem}\label{thm:main} Let $\phi:Y_-\dashrightarrow Y_+$ be the natural birational map between two $\Q$-factorial birational contractions of a Mori Dream Space $Y$. Then there exists a sharp geometric realization of $\phi$ that is a Mori Dream Space. \end{theorem}

The proof relies on the fact that a geometric realization can be defined as a projectivization of a certain subring of the Cox ring of $Y$; in particular, we may obtain it as a particular type of birational contraction --called {\em unpruning}-- of a $\P^1$-bundle over a contraction of $Y$. 
This idea allows us to present the construction of geometric realizations in an explicit way.

A substantial part of the paper has been devoted to the case of geometric realizations of {\em toric} birational transformations; in this case, our construction can be rephrased in the language of combinatorics and convex geometry. 
We present a number of software tools --programmed in the {\tt SageMath} system-- to construct and analyze geometric realizations.  These tools enable further applications, such as explicitly constructing toric geometric realizations for maps between torically embedded Mori Dream Spaces. All the scripts presented in this work, and some other useful functions constructed for dealing with toric varieties, have been made available at \cite{BOSSage}.

\subsection*{Outline} Section \ref{section:preliminaris1} is devoted to background material on $\C^*$-actions, quotients and toric varieties. In Section \ref{section:geometricreal} we deal with geometric realizations and with the proof of the Theorem \ref{thm:main}. The toric case and the scripts we have developed are presented in Section \ref{section:toric}; this section contains an example --a toric Fano fourfold-- that illustrates how our tools work. We finish the paper with Section \ref{sec:MDSs}, where we describe how to extend our explicit construction to Mori Dream Spaces, and provide an example.

\subsection*{Acknowledgements}  The first author was supported by the Institute for Basic Science (IBS-R032-D1). The second author would like to thank Cinzia Casagrande for suggesting to analyze some examples included in the work of Batyrev. Second and third author partially supported by INdAM--GNSAGA.

%%%%%%%%%%%%%%%%%%%%%%%%%%%%%%%%%%%%%%%%%%%%%%%%%%%%%%%%%%%%%%%%%%%%%%%%%%%%%
%%%%%%%%%%%%%%%%%%%%%%%%%%%%%%%%%%%%%%%%%%%%%%%%%%%%%%%%%%%%%%%%%%%%%%%%%%%%%
%%
%% END OF SECTION 1
%%
%%%%%%%%%%%%%%%%%%%%%%%%%%%%%%%%%%%%%%%%%%%%%%%%%%%%%%%%%%%%%%%%%%%%%%%%%%%%%
%%%%%%%%%%%%%%%%%%%%%%%%%%%%%%%%%%%%%%%%%%%%%%%%%%%%%%%%%%%%%%%%%%%%%%%%%%%%%

%%%%%%%%%%%%%%%%%%%%%%%%%%%%%%%%%%%%%%%%%%%%%%%%%%%%%%%%%%%%%%%%%%%%%%%%%%%%%
%%%%%%%%%%%%%%%%%%%%%%%%%%%%%%%%%%%%%%%%%%%%%%%%%%%%%%%%%%%%%%%%%%%%%%%%%%%%%
%%
%% SECTION 2 
%%
%%%%%%%%%%%%%%%%%%%%%%%%%%%%%%%%%%%%%%%%%%%%%%%%%%%%%%%%%%%%%%%%%%%%%%%%%%%%%
%%%%%%%%%%%%%%%%%%%%%%%%%%%%%%%%%%%%%%%%%%%%%%%%%%%%%%%%%%%%%%%%%%%%%%%%%%%%%
\section{Preliminaries}\label{section:preliminaris1}

%%%%%%%%%%%%%%%%%%
\subsection{Notation}\label{ssec:notn}
%%%%%%%%%%%%%%%%%%

We work over the field of complex numbers. By a \emph{polarized pair} we mean a  pair $(X,L)$, where $X$ is a normal projective variety and $L$ is an ample line bundle on $X$. 
A birational map $\phi:X\dashrightarrow Y$ is a \emph{small modification} if it is an isomorphism over an open subset whose complement has codimension at least $2$.  If $X$ and $Y$ are $\Q$-factorial, such a map is called a \emph{small $\Q$-factorial modification} (SQM, for short). 
By a {\it flip} we mean a $D$-flip as in \cite{Thaddeus1996}. The symbol $Y_{\pm}$ represents two varieties, $Y_-$ and $Y_+$, sharing a common property. 
Given  a normal projective variety $X$, and an effective $\Q$-Cartier divisor $D$  on $X$, we denote the {\it section ring of $D$} as $R(X;D)=\bigoplus_{m\geq 0} \HH^0(X,mD)$. A normal, $\Q$-factorial projective variety $X$ is called a {\it Mori Dream Space} (MDS for short) if 
\begin{enumerate}[leftmargin=\yy pt]
	\item\label{item:1} the Picard group $\Pic(X)$ of $X$ is finitely generated;
	\item\label{item:2} the nef cone $\Nef(X)$ is generated by classes of finitely many semiample divisors;
	\item\label{item:3} there exists a finite number of SQMs $f:X\dashrightarrow X_i$, for $i=0,\ldots,k$, such that every $X_i$ satisfies $(\ref{item:2})$ and $$\Mov(X)=\bigcup_{i=0}^k f^*\Nef(X_i).$$
\end{enumerate}
Let $X$ be a MDS and let $D_1,D_2$ be two effective $\Q$-Cartier divisors. The divisors $D_1$ and $D_2$ are {\it Mori equivalent} if $\Proj R(X;D_1)\simeq \Proj R(X;D_2)$. A {\it Mori chamber} in $\NU(X)$ is the closure of a Mori equivalence class whose interior is open in $\NU(X)$.

%%%%%%%%%%%%%%%%%%
\subsection{$\C^*$-action on projective varieties}\label{ssection:C*Actions}
%%%%%%%%%%%%%%%%%%

We will denote by $\C^*$ the complex $1$-dimensional algebraic torus. 
We recall the basic background on $\C^*$-actions, referring to \cite{WORS1,PhDBarban} and the references therein for more details.
Consider a $\C^*$-action on a normal projective variety $X$. The fixed point locus $X^{\C^*}$ can be decomposed into fixed connected components, i.e. $X^{\C^*}=\bigsqcup_{Y\in \cY} Y$, where  $\cY$ denotes the set of connected fixed components.

For any $x\in X$, the orbit map can be extended to a morphism $\P^1\times X\to X$, that is, there exist $\lim_{t\to 0}tx,\lim_{t\to \infty} tx$. Moreover, for any $Y\in \cY$ we define the {\it Bia\l ynicki-Birula cells} as
$$ X^+(Y)=\{x\in X\mid \lim_{t\to 0} tx\in Y\}, \quad X^-(Y)=\{x\in X\mid \lim_{t\to \infty} tx\in Y\}.$$
As a corollary of the famous result of Bia\l ynicki-Birula (see \cite{BB}), there exists a unique fixed connected component $Y_-$ (resp. $Y_+$) such that $X^-(Y_-)$ (resp. $X^+(Y_+)$) is dense. We respectively call $Y_-$ and $Y_+$ the \emph{sink} and \emph{source} of the $\C^*$-action on $X$. We refer to $Y_{\pm}$ as the \emph{extremal fixed point components} of the action; the remaining connected fixed components will be called \emph{inner}, and their set will be denoted by $\cY^\circ$.

Let $L$ be an ample line bundle on $X$: by \cite[Proposition 2.4 and subsequent Remark]{KKLV}, $L$ is $\C^*$-linearizable, hence we can consider the $\C^*$-action on the polarized pair $(X,L)$. For any $Y\in \cY$, $\C^*$ acts on the fibers of $L$ over $Y$ with  a fixed weight $\mu_L(Y)\in \Mo(\C^*)\simeq \Z$. The induced map $\mu_L:\cY\to \Z$ is called the \emph{weight map}, and $\mu_L(Y)$ is called the \emph{critical value} of $L$ on $Y$. Let $a_0,\ldots,a_r$ be the different critical values of a linearization of $L$, and order them as $a_0<a_1<\ldots<a_r$. The integer $r$ and the difference $\delta:= a_r-a_0$ are respectively called the \emph{criticality} and the \emph{bandwidth} of the $\C^*$-action on $(X,L)$. One may show that, since $L$ is ample, we have that $a_0=\mu_L(Y_-)$, $a_r=\mu_L(Y_+)$. Notice that two linearizations of $L$ differ by a character, hence $r$ and $\delta$ are independent of the chosen linearization. 

A $\C^*$-action is \emph{equalized at} $Y\in \cY$ if for every point $x\in (X^-(Y)\cup X^+(Y))\setminus Y$ the isotropy group of the $\C^*$-action at $x$ is trivial. A $\C^*$-action is called {\em equalized} if it is equalized at every component $Y\in \cY$. 

%%%%%%%%%%%%%%%%%%
\subsection{Geometric quotients and the associated birational map of a $\C^*$-action}\label{ssection:GitBirational}
%%%%%%%%%%%%%%%%%%

We refer to \cite[\S 2.2]{WORS3} for details. Let $\C^*$ act on a polarized pair $(X,L)$. Set $\cA=\bigoplus_{m\geq 0} \HH^0(X,mL)$, and let $\HH^0(X,mL)_k$ be the $\C^*$-invariant subspace where $\C^*$ acts with weight $k$.

For any $\tau\in [a_0,a_r]\cap \Q$, consider the algebra $\cA(\tau):=\bigoplus_{m\geq 0, m\tau\in\Z} \HH^0(X,mL)_{m\tau}$. By \cite[Proposition 2.11]{WORS3}, the $\C$-algebra $\cA(\tau)$ is finitely generated, and $$\GX(\tau) := \Proj \cA(\tau)$$ is a projective GIT quotient of the action. If $\tau$ is not a critical value, then the quotient $\GX(\tau)$ is geometric, and parametrizes orbits of stable points of the open, $\C^*$-invariant subset 
$$X^s_i:=X\setminus \left(\bigsqcup_{\mu_L(Y)< \tau } X^+(Y)\sqcup \bigsqcup_{\mu_L(Y)\geq \tau} X^-(Y)\right), $$ where $i$ is such that $\tau\in (a_i,a_{i+1})$.
If instead $\tau=a_i$, for a certain $i\in \{0,\ldots,r\}$, then the quotient $\GX(\tau)$ is the semigeometric quotient of the set (of semistable points) 
$$X\setminus \left(\bigsqcup_{\mu_L(Y)\leq a_i} X^+(Y)\sqcup \bigsqcup_{\mu_L(Y)\geq a_{i+1}} X^-(Y)\right).$$

For any $\tau, \tau'\in (a_{i-1},a_{i})\cap \Q$, it holds that $\GX(\tau)\simeq \GX(\tau')$; in particular we are allowed to define:
$$\GX_i:=\GX(\tau),\quad\mbox{for any }\tau\in (a_{i-1},a_{i})\cap \Q.$$ 
With this in mind, the geometric quotients $\GX_-:=\GX_1
$ $\GX_+:=\GX_{r}
$ 
are called \emph{extremal}, and the remaining geometric quotients $\GX_i$, for $i=2,\ldots,r-1$, are called \emph{inner}.

\begin{proposition}\cite[Theorem 3.6]{BRUS}\label{proposition:Pruning}
	Let $\C^*$ act on a polarized pair $(X,L)$ with criticality $r$. Let $\tau_-<\tau_+$ be two rational numbers in $[a_0,a_r]$. The $\C$-algebra
	$$\cA(\tau_-,\tau_+) = \bigoplus_{m\geq 0} \bigoplus_{k=m\tau_-}^{m\tau_+} \HH^0(X,mL)_k$$
	is finitely generated, and $X(\tau_-,\tau_+)=\Proj \cA(\tau_-,\tau_+)$ is a normal projective $\C^*$-variety, equivariantly birational to $X$, with sink $\GX(\tau_-)$ and source $\GX(\tau_+)$. \end{proposition}

Again, for rational numbers $\tau_\pm$ that are not critical values of the action, the varieties $X(\tau_-,\tau_+)$ depend only on the particular intervals between critical points containing them, so that it makes sense to denote:
$$
X_{i,j}:=X(\tau_-,\tau_+) \quad\mbox{whenever}\quad \tau_-\in (a_{i-1},a_i), \quad\tau_+\in (a_{j-1},a_j).
$$

\begin{definition}\label{def:prununprun}
With the above notation, we will say that the varieties $X_{i,j}$ are called {\em prunings of} $X$. The $\C^*$-equivariant birational maps $X\dashrightarrow X_{i,j}$ are called {\em pruning maps}, and their inverses {\em unpruning maps}.
\end{definition}

\begin{definition}\label{definition:Btype}
	A $\C^*$-action on $(X,L)$ is of \emph{B-type} if there exist isomorphisms $\GX_-\simeq \GX(a_0)\simeq Y_-$, $\GX_+\simeq \GX(a_r)\simeq Y_+$.
\end{definition}

As already observed (cf. \cite[Remark 2.3.29]{PhDBarban}), one may always assume that a $\C^*$-action on $(X,L)$ is of B-type, up to performing a pruning of $X$ at $\tau_-\in (a_0,a_1), \tau_+\in (a_{r-1},a_r)$.

\begin{definition}\label{definition:Bordism}
	A $\C^*$-action on $(X,L)$ is called a \emph{bordism} if it is of B-type, and $\overline{X^{\pm}(Y)}$ do not contain a divisor, for every $Y\in\cY^\circ$.
\end{definition}

\begin{proposition}\label{prop:assocbirmap}
	Let $\C^*$ act on a polarized pair $(X,L)$ with criticality $r$. Since $\bigcap_{i=0}^{r-1} X^s_i\neq \emptyset$, there exists a birational map
	$$\psi\colon\GX_-\dashrightarrow \GX_+,$$
	called the {\em associated birational map}, which factors through the inner geometric quotients
	\begin{equation}\label{eq:factor}
	\xymatrix@C=25pt{\GX_-\ar@{-->}^{\psi_1}[r]&\GX_2\ar@{-->}^(0.45){\psi_2}[r]&\ \ldots\ \ar@{-->}^(0.4){\psi_{r-3}}[r]&\GX_{r-1}\ar@{-->}^(0.4){\psi_{r-1}}[r]&\GX_+.}
	\end{equation}
\end{proposition}

If the action is a bordism, then the above maps are isomorphisms in codimension one.  If we assume in addition $X$ to be $\Q$-factorial and the $\C^*$-action to be a bordism equalized at the sink and the source, then $Y_{\pm}$ are $\Q$-factorial (cf. \cite[Lemma 5.3]{BRUS}), and the map $\psi$ is an SQM. 

In particular, a B-type $\C^*$-action on $(X,L)$ encodes information about a birational map $\psi:Y_-\dashrightarrow Y_+$ together with a factorization of it. In Section \ref{ssection:GeometricRealization} we will explain how, under some assumptions, we may reconstruct the $\C^*$-action upon the birational map.

%%%%%%%%%%%%%%%%%%
\subsection{Toric varieties}\label{ssection:Toric}
%%%%%%%%%%%%%%%%%%

We start by setting the notation we will use for toric varieties, and refer to \cite{CLS} for details. We have made use of the software system {\tt SageMath}, that includes a number of interesting tools to handle toric varieties; we refer the interesting reader to \cite{sagetoric} for the corresponding documentation.

Given an algebraic torus $T$, we denote by $\MM(T)$ (resp. $\NN(T)$) the lattice of characters (resp. $1$-parameter subgroups) of $T$, and set $\MM(T)_\R:=\MM(T)\otimes_\Z\R$ (resp. $\NN(T)_\R:=\NN(T)\otimes_\Z\R$). The symbol $\langle \cdot , \cdot \rangle$ denotes the nondegenerate bilinear pairing between $\MM(T)$ and $\NN(T)$, which extends to a pairing between $\MM(T)_\R$ and $\NN(T)_\R$. 

Given an $n$-dimensional projective toric variety $X$, we denote by $\Sigma\subset \NN(T)_{\R}$  
its associated fan, and write $X=X(\Sigma)$. 
We denote by $\Sigma(1)$ the set of \emph{rays} of $\Sigma$, that is the set of $1$ dimensional cones of $\Sigma$. We recall that every ray $\rho$ is associated with a $T$-invariant prime divisor $D_{\rho}$, and we abuse notation by identifying a ray $\rho$ with its primitive generator in the lattice. The fan of a projective toric variety is \emph{complete}, that is, the cones of $\Sigma$ span $\NN(T)_{\R}$. We will usually consider the case of toric varieties whose fans are \emph{simplicial}, that is, such that the generators of every cone of $\Sigma$ are linearly independent; this is equivalent to say that the variety $X(\Sigma)$ is $\Q$-factorial. 

 Let $X(\Sigma)$ be a a complete toric variety and let  $\Sigma(1)=\{\rho_1,\ldots,\rho_k\}$ be the set of rays of the fan $\Sigma$. Let $b=(b_1,\ldots,b_k)$  be a vector in $\R^k$, and let $D=\sum_{i=1}^k b_iD_{\rho_i}$ be a $T$-invariant $\Q$-Cartier divisor on $X$.  The {\em moment polytope} $P_D$ associated with the pair $(X,\cO_X(D))$ is   
$$P_D = \{ m \in \MM(T)_{\R} \mid \langle \rho_i,m\rangle + b_i \geq 0, \quad i=1,\ldots,k\}.$$
We will essentially use moment polytopes  of {\em big} divisors, which represent birational contractions of toric varieties. 

The  {\tt SageMath} system contains a number of powerful tools to deal with toric varieties and associated combinatorial/convex-geometric objects, such as moment polytopes. Besides the scripts already included in the current version of {\tt SageMath} we have written a number of useful functions in the framework of toric varieties (which we could not find in {\tt SageMath} libraries), that were particularly important for our goals. Among those available in \cite{BOSSage} we highlight the following.

\begin{itemize}[leftmargin=\xx pt,itemsep=5pt]

\item {\tt\verb|poly(rays,D)|}: Computes the moment polytope of a torus invariant divisor $D$ on a projective toric variety. In the input of the function, {\tt rays} is the list of the rays $D_{\rho_i}$, $i=1,\dots,k$ of the fan $\Sigma$ defining $X$, and the divisor $D=\sum_{i=1}^k b_iD_{\rho_i}$ is introduced as the vector $(b_1,\dots,b_k)$.

\item {\tt\verb|projective_bundle(X,M)|}: Given a toric variety $X$ and a matrix $M$, whose rows contain the coefficients of some divisors $D_1,\dots,D_t$ in $X$, the function returns the toric variety $\P_X(\bigoplus_{i=1}^t\cO_X(D_i))$. The script is based on \cite[Section~7.3]{CLS}, and uses the function {\tt\verb|proj_rays(X,M)|}, that computes the rays of the fan of $\P_X(\bigoplus_{i=1}^t\cO_X(D_i))$.

\item {\tt\verb|toric_from_PR(PR)|}: Constructs a toric variety from the set of primitive relations, given as a list of lists.

\item {\tt\verb|mov_cone(X)|}: Computes the movable cone of the toric variety $X$, using the description given in \cite[Proposition 15.2.4]{CLS}.

\item {\tt\verb|secondary_fan(X)|}: Computes the rays and chambers of the secondary fan  of the toric variety $X$ (see \cite[Proposition 15.2.1]{CLS}).

\end{itemize}

More specialized functions, to deal with the birational geometry of a toric variety are the following:

\begin{itemize}[leftmargin=\xx pt,itemsep=5pt]
\item {\tt\verb|same_chamber(X,A,B)|}:
Given Cartier divisors $A$ and $B$ on a toric variety $X$, the function returns a boolean indicating whether the two divisors are Mori equivalent. To do so, the function tests whether the polytope associated with the divisor $A+B$, which corresponds to the {\it Minkowski sum} of the polytopes  $P_A$ and $P_B$, is combinatorially equivalent to both the polytopes $P_A$ and $P_B$. 

\item {\tt\verb|modify(X,A)|}: Construct another random $\Q$-divisor $B$ on $X$ belonging to the same chamber as the divisor $A$. 

\item {\tt\verb|is_wall_crossing(PA,PB)|}:
Given polytopes $P_A$ and $P_B$, associated with big Cartier divisors $A$ and $B$ belonging to the interior of two different maximal dimensional Mori chambers of a toric variety, the function determines if the two chambers share a common wall, comparing the numbers of the facets of $P_A, P_B$ and $P_{A+B}$.
\end{itemize}\par\medskip

Finally, we will need some tools that provide information about the action of a $1$-dimensional subtorus of $T$ on the toric variety $X$, polarized with an ample divisor $D$, with moment polytope $P=P_D$. For simplicity, we will consider a fixed isomorphism $T\simeq (\C^*)^{\dim(X)}$, and consider  the $\C^*$-action given by the inclusion $\C^*\hookrightarrow T$ corresponding to the projection $\mu_j:\Mo(T)\to\Mo(\C^*)$  to the $j^{\mbox{\scriptsize th}}$-coordinate. 

\begin{itemize}[leftmargin=\xx pt,itemsep=5pt]

\item {\tt\verb|sort_fixed_points(P,j)|}: Returns the weights, the number of vertices of the polytopes of the different connected fixed points components of $X^{\C^*}$ and the criticality of the action (partial outputs can be obtained with the functions {\tt\verb|weights(P,j)|}, {\tt\verb|fixedpoints(P,j)|}, {\tt\verb|criticality(P,j)|}).

\item {\tt\verb|pruning(P,j,a,b)|}: Computes the pruning of $X$  at $a,b$ described in Proposition \ref{proposition:Pruning}. This actions takes as input the polytope $P$, an integer $j$ corresponding to the coordinate determining the action, and two rational numbers $a,b$. 

\item {\tt\verb|quotient(P,j,a)|}: Given a rational number $a$, this function returns the GIT-quotient $\GX(a)$ of $X$ by the $\C^*$-action defined by the integer $j$, obtained by using the pruning function previously described, with the values  $a=b$.

\item {\tt\verb|geometric_quotients(P,j)|}:
Returns the list of all geometric quotients of $X$ with respect to the action corresponding to the projection $\mu_j:\Mo(T)\to\Mo(\C^*)$  to the $j^{\mbox{\scriptsize th}}$-coordinate.

\end{itemize}

%%%%%%%%%%%%%%%%%%%%%%%%%%%%%%%%%%%%%%%%%%%%%%%%%%%%%%%%%%%%%%%%%%%%%%%%%%%%%
%%%%%%%%%%%%%%%%%%%%%%%%%%%%%%%%%%%%%%%%%%%%%%%%%%%%%%%%%%%%%%%%%%%%%%%%%%%%%
%%
%% END OF SECTION 2
%%
%%%%%%%%%%%%%%%%%%%%%%%%%%%%%%%%%%%%%%%%%%%%%%%%%%%%%%%%%%%%%%%%%%%%%%%%%%%%%
%%%%%%%%%%%%%%%%%%%%%%%%%%%%%%%%%%%%%%%%%%%%%%%%%%%%%%%%%%%%%%%%%%%%%%%%%%%%%

%%%%%%%%%%%%%%%%%%%%%%%%%%%%%%%%%%%%%%%%%%%%%%%%%%%%%%%%%%%%%%%%%%%%%%%%%%%%%
%%%%%%%%%%%%%%%%%%%%%%%%%%%%%%%%%%%%%%%%%%%%%%%%%%%%%%%%%%%%%%%%%%%%%%%%%%%%%
%%
%% SECTION 3 
%%
%%%%%%%%%%%%%%%%%%%%%%%%%%%%%%%%%%%%%%%%%%%%%%%%%%%%%%%%%%%%%%%%%%%%%%%%%%%%%
%%%%%%%%%%%%%%%%%%%%%%%%%%%%%%%%%%%%%%%%%%%%%%%%%%%%%%%%%%%%%%%%%%%%%%%%%%%%%
\section{Geometric realizations of birational maps}\label{section:geometricreal}

Let us start from a birational map $\phi: Y_-\dashrightarrow Y_+$. We will consider factorizations of $\phi$ into ``elementary'' birational transformations; more concretely, we will be interested in the following types of factorizations: 

\begin{definition}\label{def:sharp}
Given a birational map $\phi: Y_-\dashrightarrow Y_+$ between normal $\Q$-factorial projective varieties, a factorization of $\phi$: 
\[
\xymatrix{Y_-\ar@{=}[r]\ar@<1ex>@/^1.3pc/@{-->}[rrrrrr]^{\phi}&Y_1\ar@{-->}_{\phi_1}[r]&Y_2\ar@{-->}_{\phi_2}[r]&\ldots \ar@{-->}_{\phi_{s-3}}[r]&Y_{s-2}\ar@{-->}_{\phi_{s-1}}[r]&Y_s\ar@{=}[r]&Y_+.}
\] 
as a composition of birational maps $\phi_i$ is called {\em directed} if there exist a $\Q$-factorial variety $Y$, Cartier divisors $A,B$ on $Y$ and positive rational numbers $\tau_2<\tau_3<\dots<\tau_{s-1}$, such that: 
\begin{itemize}[leftmargin =\xx pt]
\item $Y_1=\Proj R(Y;A)$, $Y_s=\Proj R(Y;B) $, and $Y_k=\Proj R(Y;A+\tau_{k}B)$ for every $k=2,\dots,s-1$;
\item for every $k=1,\dots,s-1$, we have $\phi_k=\rho_{k+1}\circ \rho_k^{-1}$, where  $\rho_k$ denotes the natural map $Y\dashrightarrow Y_k$. 
\end{itemize}
If moreover the maps $\phi_k$, for $k=1,\dots,s-1$, are either elementary divisorial contractions, extractions, or flips (corresponding to wall-crossings in the Mori chamber decomposition of the effective cone of the corresponding varieties), then we say that the factorization is {\em sharp}.
\end{definition} 

\begin{remark}\label{rem:sharp}
If the map $\phi$ is a small $\Q$-factorial modification between two Mori Dream Spaces, then a directed resolution of $\phi$ always exists. We 
identify $\NU(Y_-)$ with $\NU(Y_+)$, and we consider classes of Cartier divisors $A,B \in \NU(Y_-)$ such that $A$ is ample on $Y_-$ and $B$ is ample on $Y_+$. The segment joining $A$ and $B$ in $\NU(Y_-)$ intersects a finite number of facets of the Mori chambers of $Y_-$, that correspond to small modifications $Y_k$ of $Y_-$, and that provides the directed factorization of $\phi$.  
Moreover, by taking $A,B$ general in the ample cones of $Y_-,Y_+$, the factorization that we get will be sharp.
\end{remark}

We want to construct $\C^*$-actions that have $\phi$ as associated birational map. More concretely, we recall the following concept, that has been introduced in \cite{WORS4}:

\begin{definition}\label{definition:GeometricRealization}
	Let $\phi: Y_-\dashrightarrow Y_+$ be a birational map between normal projective varieties. A \emph{geometric realization} of $\phi$ is a normal projective variety $X$, endowed with a $\C^*$-action of B-type, whose associated birational map (see Proposition \ref{prop:assocbirmap}) coincides with $\phi$. A geometric realization of $\phi$ determines a factorization (\ref{eq:factor}) of the map, which is always directed; if the factorization is sharp, we say that the geometric realization is {\em sharp}.
\end{definition}

%%%%%%%%%%%%%%%%%%
\subsection{Geometric realizations of maps of dream type}\label{ssection:GeometricRealization}
%%%%%%%%%%%%%%%%%%

The problem of existence  of geometric realizations of small $\Q$-factorial modifications has been dealt with in \cite{BRUS}. In this section we will briefly recall the results presented there.

\begin{definition}\label{definition:SmallModificationDreamType}
	Let $f \colon Y_-\dashrightarrow Y_+$ be a small modification between normal projective varieties $Y_{\pm}$. The map $f$ is \emph{of dream type} if there exist effective Cartier divisors $A,B$ on $Y_-$ such that:
	\begin{enumerate}[leftmargin=\yy pt]
		\item $A$ is ample;
		\item $Y_+=\Proj \bigoplus_{m\geq 0} \HH^0(Y_-,mB)$;
		\item the multisection ring $$R(Y_-;A,B):= \bigoplus_{a,b\geq 0} \HH^0(Y_-,aA+bB)$$ is a finitely generated $\C$-algebra.
	\end{enumerate}	
\end{definition}

The following statement shows the existence of geometric realizations under the dream type assumption introduced above. 

\begin{theorem}\cite[Theorem 4.1, Corollary 4.10]{BRUS}\label{theorem:Brus}
	Let $\phi: Y_-\dashrightarrow Y_+$ be a small modification of dream type. Then there exists a geometric realization of $\phi$ which is a bordism equalized at the sink and the source.
\end{theorem} 

Roughly speaking, the geometric realization $X$ of a small modification provided by the above statement is constructed by considering an appropriately chosen projective GIT $\C^*$-quotient of the spectrum of a multisection ring $R(Y_\pm;A,B)$. It comes with a linearization on an ample line bundle $L$ satisfying that:
\begin{itemize}[leftmargin=\xx pt]
\item $R(X;L)$ is a graded subalgebra of $R(Y_\pm;A,B)$. 
\item The restrictions of $L$ to the sink and the source are isomorphic to multiples of $A,B$, respectively.
\item Denoting by $a_0=0<a_1<\dots<a_r=\delta$ the critical values of the action on $(X,L)$ --so that the $\C^*$-action on $X$ has bandwidth $\delta$ and criticality $r$--, the cone $\cC = \langle L_{\mid Y_-}, L_{\mid Y_-}-\delta Y_{-\mid Y_-}\rangle$ admits a decomposition $\cC = \bigcup_{i=0}^{r-1} \cC_i$, where $\cC_i$ is the interior of $\langle L_{\mid Y_-}-a_{i-1} Y_{-\mid Y_-}, L_{\mid Y_-}-a_{i} Y_{-\mid Y_-}\rangle$ and for any $D\in \cC_i$ it holds that $\Proj R(X;D)\simeq \GX_i$. 
\end{itemize}

In other words, if $\tau_i\in (a_{i-1},a_i)$, then 
\[\GX_i=\Proj(R(Y_-;L_{\mid Y_-}-\tau_{i} Y_{-\mid Y_-})) \]
The geometric realization $(X,L)$ provides then a factorization of $\phi$:
\[\xymatrix{Y_-\ar@{=}[r]\ar@<1ex>@/^1.3pc/@{-->}[rrrrrr]^{\phi}&\GX_1\ar@{-->}_{\phi_1}[r]&\GX_2\ar@{-->}_{\phi_2}[r]&\ldots \ar@{-->}_{\phi_{s-3}}[r]&\GX_{s-2}\ar@{-->}_{\phi_{s-1}}[r]&\GX_s\ar@{=}[r]&Y_+.}
\]

As already explained in Proposition \ref{proposition:Pruning}, we may then consider the pruning $X_{i,j}$ of $X$ at $i,j$, that is the projective spectrum of the algebra
\[
\bigoplus_{m\geq 0}\bigoplus_{m\tau_-\leq k\leq m\tau_+}\HH^0(X,mL)_k\subset R(X;L),
\]
with $\tau_- \in (a_i,a_{i+1})\cap \Q$, $\tau_+\in(a_j,a_{j+1})\cap \Q$. Recall that the pruning $X_{i,j}$ is birational to $X$, and  inherits a $\C^*$-action of  criticality smaller than or equal to $r$, with sink $\GX_i$ and source $\GX_j$. In the particular case in which $\tau_-,\tau_+$ belong to the same interval $(a_{i-1},a_{i})$, then $X_{i,i}$ is a $\P^1$-bundle over the variety $\GX_i$.  

This construction yields the following idea --that we will use later on in situations different from the one of small $\Q$-factorial modifications considered above--: the geometric realization $X$ can be obtained as an unpruning of the $\P^1$-bundle $X_{i,i}$.

%%%%%%%%%%%%%%%%%%
\subsection{Proof of Theorem \ref{thm:main}}\label{ssec:unpruning}
%%%%%%%%%%%%%%%%%%

The goal of this section is to prove the existence of geometric realizations  for birational maps between two birational contractions of an MDS. In other words, here we will consider the following situation:

\begin{setup}\label{set:birational resolution}
Let $Y$ be a projective normal $\Q$-factorial variety, together with two birational contractions to two normal $\Q$-factorial projective varieties $Y_\pm$:
\[
\xymatrix{&&Y\ar@{-->}[lld]_{\pi_-}\ar@{-->}[rrd]^{\pi_+}&&\\Y_-\ar@{-->}[rrrr]^{\phi:=\pi_+\circ \pi_-^{-1}}&&&&Y_+}
\]
We will assume that $Y$ is an MDS (although the construction below of geometric realizations works under milder conditions similar to Definition \ref{definition:SmallModificationDreamType}). We will denote by $A,B\in \Pic(Y)$ two big Cartier divisors on $Y$ supporting the birational contractions $\pi_-,\pi_+$, respectively. 
\end{setup}

\begin{construction}\label{cons:geomreal}
In the above Setup, we choose a positive integer $\ell$ such that
\begin{equation*} 
%\label{eq:H}
H:=\dfrac{1}{\ell}(B-A) 
\end{equation*} 
is a Cartier divisor,
and a sufficiently large integer $m$ such that every nonempty intersection of the set $\{mA+tH|\,\,0\leq t\leq m\ell\}$ with the interior of a Mori chamber of $Y$ contains at least two elements of the form $mA+kH$, $k\in \Z$. We then consider the $\P^1$-bundle:
$$W:=\P_Y(\cO_Y\oplus \cO_Y(H)),$$
which is a $\Q$-factorial variety of Picard number $\rho_Y+1$, and supports a natural fiberwise equalized  $\C^*$-action whose sink and source are the sections $D_-,D_+$ of the natural projection $\pi:W\to Y$ corresponding to the quotients $\cO_Y\oplus \cO_Y(H)\to \cO$, $\cO_Y\oplus \cO_Y(H)\to \cO(H)$, respectively.  We will define $X$ in the following way:
\begin{equation}\label{eq:gr}
X:=\Proj R\big(W;m(\ell L+\pi^*A)\big),
\end{equation}
where $L$ denotes a divisor such that $\cO_W(L)$ is the tautological line bundle $\cO_W(1)$.
\end{construction}

The proper definition of $X$ is justified as follows. Note first that by definition $\ell L+\pi^*A$ is a big divisor in $W$. Since we have assumed that $Y$ is an MDS, it follows from \cite[Theorem~3.2]{Bro} that $W$ is an MDS as well. In particular the $\C$-algebra $R\big(W;m(\ell L+\pi^*A)\big)$ is finitely generated, and its projectivization $X$ is well defined and normal. It is a birational contraction of $W$.

Note also that the variety $X$ constructed above inherits a $\C^*$-action of B-type, with sink $\Proj R(Y;mA)=Y_-$ and source $\Proj R(Y;m(\ell H+A))=Y_+$, and whose associated birational map is $\phi=\pi_+\circ \pi_-^{-1}$. In other words:

\begin{proposition}\label{prop:geomrealexists}
In the situation of Setup \ref{set:birational resolution} the variety $X$ defined in Construction \ref{cons:geomreal} is a geometric realization of $\phi: Y_-\dashrightarrow Y_+$. \qed
\end{proposition}

\begin{remark}\label{rem:sharp}
By slightly modifying $A$, $B$ within the corresponding Mori chambers of $Y$, we may assume that the geometric realization $X$ is sharp.
\end{remark}

\begin{remark}\label{rem:geomrealunprun} 
Given a positive integer $k\leq m\ell-1$ such that the classes of the divisors $mA+kH$, $mA+(k+1)H$ belong to the same Mori chamber of $Y$, then the pruning of $X$: 
\[
\begin{split}
W':=&\Proj \bigoplus_{r\geq 0}\bigoplus_{rk\leq j\leq r(k+1)} \HH^0(W,rm(\ell L+\pi^*A))_j=\\
=&\Proj \bigoplus_{r\geq 0}\bigoplus_{rk\leq j\leq r(k+1)}\HH^0\big(Y,\cO_{Y}(rmA+jH)\big)=\\
=& \Proj \bigoplus_{r\geq 0}\HH^0\big(Y,S^r\big(\cO_{Y}(mA+kH)\oplus\cO_{Y}(mA+(k+1)H)\big)\big)
\end{split}
\]
is a $\P^1$-bundle over the variety $Y':=\Proj R(Y;mA+kH)$, which is a birational contraction of $Y$.  

Let us now assume that there exists $k$ such that $mA+kH,mA+(k+1)H$ belong to the ample cone of $Y$ so that $W'\simeq W$: alternatively, one may substitute $Y$ by a birational contraction $\Proj R(Y;mA+kH)$ of maximal Picard number among all the birational contractions $\Proj R(Y;mA+sH)$, $s=0,\dots,m$.
Then we see that $X$ is an unpruning of the $\P^1$-bundle $W$. Note also that, setting $L':=L+\pi^*(mA+kH)$, $X$ can be written as
$$
X=\Proj R(W; L'+kD_-+(m\ell-k-1)D_+). 
$$ 
\end{remark}

\begin{proposition}\label{prop:Qfact}
In the situation of Setup \ref{set:birational resolution}, choose the divisors $A,B$ so that the geometric realization $X$ is sharp. Then $X$ is an MDS. 
\end{proposition}

\begin{proof}
Since $X$ is a birational contraction of $W$, which is an MDS, it is enough to show that $X$ is $\Q$-factorial (cf. \cite[Proposition~1.11]{HuKeel}). Moreover, it suffices to prove that  $X$ is $\Q$-factorial locally (in the Zariski topology) around every $\C^*$-fixed point of $X$. Indeed, if $X$ is not locally $\Q$-factorial in a neighborhood of a point $P$, then $X$ is not $\Q$-factorial in a neighborhood of every point $Q\in \overline{\C^*\cdot P}$.

If $P\in X$ is an extremal fixed point of $X$ --for instance, if it belongs to the sink of $X$--, then $X$ is locally isomorphic (around $P$) to a $\P^1$-bundle over the $\Q$-factorial variety $\GX_{1}$; then it is $\Q$-factorial at $P$. If $P$ is an inner fixed point, say of weight $a_i$, then $X$ is locally isomorphic around $P$ to the pruning $X_{i-1,i+1}$, which is a geometric realization of the birational map $Y_i=\GX_{i}\dashrightarrow Y_{i+1}=\GX_{i+1}$; since we are assuming that $X$ is a sharp geometric realization, this map corresponds to a wall-crossing in the effective cone of $Y$. 

This argument shows that, without loss of generality, we may assume that the criticality of the action  is $r=2$, the action is of $B$-type, and the birational map from the sink to the source $\phi:Y_-\dashrightarrow Y_+$ is a wall-crossing.  We have now three possibilities: either $\phi$ is a flip, a divisorial contraction, or a divisorial extraction. The third case reduces to the second by inverting the action of $\C^*$, so we may assume that $\phi$ is a flip or a divisorial contraction. We will denote by $Z\subset X$ the closed subset of inner fixed points of the action. In both cases, we will consider the pruning $W':=X_{0,1}$ which, as in Remark \ref{rem:geomrealunprun}, is a $\P^1$-bundle over $Y_-=\GX_{1}$, of the form $\P(\cO_{Y_-}\oplus\cO_{Y_-}(H'))$ for a certain Cartier divisor $H'$ on $Y_-$; in particular, by \cite[Theorem~3.2]{Bro}, $W'$ is an MDS. Let us denote by $\pi':W'\to Y_-$ the natural projection, by $L'$ a divisor associated with the tautological line bundle $\cO_{W'}(1)$, and by $D'_+\subset W'$ the section associated with the surjection $\cO_{Y_-}\oplus\cO_{Y_-}(H')\to \cO_{Y_-}(H')$. In this situation, $Y_+$ corresponds to a Mori chamber in $\ol{\Eff(Y_-)}$, adjacent to the nef cone of $Y_-$. 

As an unpruning of $W'$, $X$ can be written as $\Proj R(W';L'+\pi^*A'+kD'_+)$, for a line bundle $A'$ on $Y_-$ and a positive integer $k$ such that $(L'+\pi^*A'+kD'_+)_{|D'_+}=(k+1)H'+A'$ belongs to the interior of the Mori chamber of $Y_+$. In particular $\phi$ can be identified with the natural map $Y_-\to \Proj R(Y_-;(k+1)H'+A')$, which is a flip or a divisorial contraction corresponding  to the crossing of the wall defined by an effective $1$-cycle $C$ in $Y_-$. In particular, there exists a positive rational $q\leq k$ such that $(qH'+A')\cdot C=0$. Identifying $C$ with the corresponding $1$-cycle in the section $D'_+\subset W$, we then have that $L'+\pi^*A'+(q-1)D'_+$ is a nef divisor in $W'$, supporting the contraction of the ray generated by the class of $C$, and $L'+\pi^*A'+kD'_+$ supports the crossing of the wall defined by $C$. In particular, since $W'$ is an MDS, we conclude that also $\Proj R(W';L'+\pi^*A'+kD'_+)$ is an MDS, hence $\Q$-factorial. 
\end{proof}

%%%%%%%%%%%%%%%%%%
\subsection{When are geometric realizations Fano?}\label{ssec:fanogeomreal}
%%%%%%%%%%%%%%%%%%

Let us observe that, in the situation of Setup \ref{set:birational resolution}, with the notation of the previous section, a geometric realization $X$ of the birational map $\phi$ is constructed as a birational contraction of a $\P^1$-bundle $W=\P(\cO_Y\oplus \cO_Y(H))$. The anticanonical divisor of $W$ is:
\[
-K_{W}=D_-+D_+-\pi^*K_Y.
\]

By hypothesis, $Y$ is an MDS; we will further assume in this section that $Y$ is ($\Q$-Gorenstein) {\em Fano}, in the sense that its anticanonical divisor is ample.

\begin{remark}\label{rem:BCHM}
Note that if $Y$ is a $\Q$-factorial Fano variety with log terminal singularities, then the MDS condition of Setup \ref{set:birational resolution} follows from \cite[Corollary~1.3.1]{BCHM}. 
\end{remark}

By the above formula, if $-K_Y$ is ample, then $-K_W$ is effective, the Kodaira--Iitaka dimension of $-K_W$ is at least $\dim Y=\dim X-1$, and equality holds  when $\Proj R(W;-K_W)=Y$. If this is not the case, i.e., if $-K_W$ is big, then a  birational contraction of $W$ will be Fano, as well. It then makes sense to ask whether this birational contraction can be obtained as an unpruning of $W$, so that it will be a geometric realization of a birational map $\phi$ between two birational models $Y_\pm$ of $Y$. 

Note that we may write $-K_{W}=2L-\pi^*(K_Y+H)=2(L-\pi^*(K_Y+H)/2)$, and $L-\pi^*(K_Y+H)/2$ is the tautological divisor of the $\Q$-twisted bundle $(\cO_Y\oplus\cO_Y(H))\langle -\pi^*(K_Y+H)/2\rangle$. Now this is big if --up to change of sign of $H$-- 
\[
-K_Y-H \mbox{ is big, and } -K_Y+H \mbox{ is effective.}
\]
For simplicity, we will consider  the case in which both $-K_Y\pm H$ are big, and  belong to the interior of the Mori chambers of $\ol{\Eff(W)}$. 

In the case in which $-K_Y\pm H$ are ample, then $W$ is a Fano $\P^1$-bundle; then the arguments here  can be considered as an extension of the results on Fano bundles, which have been extensively studied in the literature as a way of constructing Fano varieties (see, for instance \cite{MOS4,SSW,SW1,SW2,APW}).

For an integer $m$ large enough,  $-(mK_Y+H)$ is ample in $Y$, so that the divisor
$L':=L-\pi^*(mK_Y+H)$ is ample in $W$. Moreover we may write:
\[
\begin{split}
L'+(m-1)D_-+m D_+&=L-\pi^*(mK_Y+H)+(m-1)(L-\pi^*H)+mL=\\&=m(2L-\pi^*(K_Y+H))=-mK_W;
\end{split}
\]
this tells us that the unpruning  $X:=\Proj R(W;L'+(m-1)D_-+m D_+)$ is equal to $\Proj R(W,-K_W)$ which is a Fano variety.

%%%%%%%%%%%%%%%%%%%%%%%%%%%%%%%%%%%%%%%%%%%%%%%%%%%%%%%%%%%%%%%%%%%%%%%%%%%%%
%%%%%%%%%%%%%%%%%%%%%%%%%%%%%%%%%%%%%%%%%%%%%%%%%%%%%%%%%%%%%%%%%%%%%%%%%%%%%
%%
%% END OF SECTION 3
%%
%%%%%%%%%%%%%%%%%%%%%%%%%%%%%%%%%%%%%%%%%%%%%%%%%%%%%%%%%%%%%%%%%%%%%%%%%%%%%
%%%%%%%%%%%%%%%%%%%%%%%%%%%%%%%%%%%%%%%%%%%%%%%%%%%%%%%%%%%%%%%%%%%%%%%%%%%%%

%%%%%%%%%%%%%%%%%%%%%%%%%%%%%%%%%%%%%%%%%%%%%%%%%%%%%%%%%%%%%%%%%%%%%%%%%%%%%
%%%%%%%%%%%%%%%%%%%%%%%%%%%%%%%%%%%%%%%%%%%%%%%%%%%%%%%%%%%%%%%%%%%%%%%%%%%%%
%%
%% SECTION 4
%%
%%%%%%%%%%%%%%%%%%%%%%%%%%%%%%%%%%%%%%%%%%%%%%%%%%%%%%%%%%%%%%%%%%%%%%%%%%%%%
%%%%%%%%%%%%%%%%%%%%%%%%%%%%%%%%%%%%%%%%%%%%%%%%%%%%%%%%%%%%%%%%%%%%%%%%%%%%%
\section{Geometric realizations of toric birational maps}\label{section:toric}

 In this section we will work under the following hypotheses:
 
 \begin{setup}\label{set:toric}
 Let $Y$ be a projective normal $\Q$-factorial toric variety, together with two birational toric contractions to projective normal $\Q$-factorial varieties $Y_\pm$:
\[
\xymatrix@R=25pt@C=35pt{&Y\ar@{-->}[ld]_{\pi_-}\ar@{-->}[rd]^{\pi_+}&\\Y_-\ar@{-->}[rr]^{\phi:=\pi_+\circ \pi_-^{-1}}&&Y_+}
\]
We consider two big Cartier divisors $A,B\in \Pic(Y)$ supporting the birational contractions $\pi_-,\pi_+$, respectively. 
The varieties $Y_\pm$ are then toric, and the birational maps $\pi_-,\pi_+,\phi$ are torus-equivariant. 
 \end{setup}

Thanks to Theorem \ref{thm:main}, we know that the geometric realization $X$ of Construction \ref{cons:geomreal} is a MDS; furthermore we have that:

\begin{proposition}\label{prop:toricgeomreal}
	In the situation of Setup \ref{set:toric}, the geometric realization $X$ of $\phi$, given in Construction \ref{cons:geomreal}, is a projective normal $\Q$-factorial toric variety.
\end{proposition}

\begin{proof}
Following Construction \ref{cons:geomreal}, the geometric realization $X$ can be described as a birational contraction of a $\P^1$-bundle over $Y$ of the form $W=\P_Y(\cO_Y \oplus \cO_Y(H))$. Since $W$ is a toric variety, the conclusion follows.  
\end{proof}

In the toric setting, Construction \ref{cons:geomreal} can be translated into combinatorial/convex--geometric language, and programmed conveniently  to produce examples of geometric realizations. Working in this direction, we have produced the {\tt SageMath} function {\tt geometric\_realization}, that we describe hereafter. The input of the function consists of the following:
\begin{itemize}[leftmargin=\xx pt]
	\item the toric variety $Y$;
	\item the Cartier divisors $A$ and $B$ supporting the contractions $\pi_{\pm}$.
	\item an integer $\ell$ such that $(B-A)/\ell$ is a Cartier divisor. The default is $\ell=1$.
\end{itemize} 
This function mimics Construction \ref{cons:geomreal}: after setting $H:=(B-A)/\ell$, it computes the rays of the fan of the toric variety $W=\P_Y(\cO_Y \oplus \cO_Y(H))$. To do so, it uses the function  {\tt proj\_rays}, which is part of the function {\tt \verb|projective_bundle|}.

We then construct the vector corresponding to the divisor $D=\ell L+\pi^*A$ described in (\ref{eq:gr}) and the polytope $P_D$ associated with the pair $(W,\cO_W(D))$. Next, we compute an integer $m$ such that $mP_D$ is a lattice polytope, by clearing denominators in the coordinates of the vertices. The function returns this last polytope.\par\medskip

\noindent\begin{minipage}{\textwidth}
\begin{python}
def geometric_realization(Y,A,B,l=1):
    H=(B-A)/l
    r=[list(ray) for ray in (Y.fan()).rays()]
    M=Matrix([[0]*len(H),H])
    pr=proj_rays(r,M)
    L=vector(QQ,[1 if i == len(pr)-2 else 0 for i in range(len(pr))])
    pullback_A=A.concatenate(vector([0,0]))
    m=lcm([x.denominator() for x in weights(poly(pr,l*L+pullback_A))])
    return poly(pr,m*(l*L+pullback_A))
\end{python}
\end{minipage}\par\medskip

To construct a sharp realization, we define the function {\tt \verb|is_sharp|} which checks whether a realization is sharp by checking the birational type of the birational maps between two successive geometric quotients using the function {\tt \verb|is_wall_crossing|}, described in Section \ref{ssection:Toric}.\par\medskip
 
\noindent\begin{minipage}{\textwidth}
\begin{python}
def is_sharp(G,a):
    r=criticality(G,a)-1
    if r==0:
        return true
    Q=[geometric_quotients(G,a)[i] for i in range(r+1)]
    return all(is_wall_crossing(Q[i],Q[i+1]) for i in range(r))
\end{python}
\end{minipage}\par\medskip

As noticed in Remark \ref{rem:sharp}, by slightly moving the rays generated by the divisor $A$ we may find a sharp geometric realization of the map $\phi$.  This is precisely what the function {\tt \verb|sharp_realization(Y,A,B)|} does: after  computing a geometric realization (with $\ell=1$), it checks if it is sharp. If not, the starting divisor is adjusted (in the same chamber) using the function {\tt modify(Y,A)}, described in Section \ref{ssection:Toric}, and a new realization is computed, until a sharp realization is obtained.\par\medskip

\noindent\begin{minipage}{\textwidth}
\begin{python}
def sharp_realization(Y,A,B):
    n,D = 0,[A]
    GR=[geometric_realization(Y,A,B)]
    a=(GR[0]).dimension()
    while not is_sharp(GR[n],a):
            n+=1
            D.append(modify(Y,D[n-1]))
            GR.append(geometric_realization(Y,D[n],B))
    return GR[n]
\end{python}
\end{minipage}

%%%%%%%%%%%%%%%%%%%%%%
\subsection{Unpruning} \label{ssec:unprun}
%%%%%%%%%%%%%%%%%%%%%%

As we have seen in Remark \ref{rem:geomrealunprun}, geometric realizations can be described as unprunings of $\P^1$-bundles. When dealing with the birational geometry of  geometric realizations it is useful, (see, for instance, Section \ref{ssec:fanogeomreal}) to have a tool to produce  unprunings of a given geometric realization. The function {\tt unpruning(Y,E,F,a,b)} takes as input a toric variety $Y$, two Cartier divisors $E$ and $F$, and two nonnegative integers $a,b$. It first constructs the projective bundle $\P_Y(\cO(E) \oplus \cO(F))$, then the sections $D_1$ and $D_2$ corresponding to the quotients $\cO(E) \oplus \cO(F) \to \cO(E)$ and $\cO(E) \oplus \cO(F) \to \cO(F)$ and the tautological line bundle $L$. It finally returns the polytope corresponding to the unpruning divisor $L +aD_1+bD_2$.\par\medskip

\noindent\begin{minipage}{\textwidth}
\begin{python}
def unpruning(Y,E,F,a,b):
    r=[list(ray) for ray in (Y.fan()).rays()]
    M=matrix([E,F])
    pr=proj_rays(r,M)
    n=len(pr)
    D1=vector(QQ, [1 if i == n-2 else 0 for i in range(n)])
    D2=vector(QQ, [1 if i == n-1 else 0 for i in range(n)])
    L=M.row(0).concatenate(vector([0,0]))+D1
    return poly(pr,L+a*D1+b*D2)
\end{python}
\end{minipage}\par\medskip

The function above constructs a Fano realization in the toric case, following the steps outlined in  Section \ref{ssec:fanogeomreal}. First we compute an integer $m$ such that the line bundle  $-(mK_Y+H)$ is ample on $Y$ with the function  {\tt \verb|compute_m(Y,H)|}. 
\par\medskip
\noindent\begin{minipage}{\textwidth}
\begin{python}
def compute_m(Y,H):
    l=len(Y.fan().rays())
    m=1
    D=vector([1]*l)+vector(H)
    while not is_ample(Y,D):
        m += 1
        D=vector([m]*l)+vector(H)
    return m
\end{python}
\end{minipage}\par\medskip

Then we use the function {\tt \verb|Fano_realization(Y,H)|} to compute the proper unpruning as described in Section \ref{ssec:fanogeomreal}.
\par\medskip
\noindent\begin{minipage}{\textwidth}
\begin{python}
def Fano_realization(Y,H):
    l=len(Y.fan().rays())
    m=compute_m(Y,H)
    K=vector([1]*l)
    return unpruning(Y,m*K-H,m*K,m-1,m)
\end{python}
\end{minipage}

%%%%%%%%%%%%%%%%%%%%%%
\subsection{An example: representing birational contractions of a Fano fourfold}\label{ssec:Batyrev}
%%%%%%%%%%%%%%%%%%%%%%

In order to illustrate the construction of geometric representations of toric birational maps
we have chosen to consider the set of birational contractions of a smooth Fano fourfold $Y$ of Picard number three, more precisely the number $33$ in Batyrev's list (cf. \cite[Proposition~3.1.2~(iii)]{batyrev}). This is defined by a Fano polytope with $7$ vertices $v_1,\dots,v_7\in \NN((\C^*)^4)$ satisfying the following primitive relations:
\[\begin{array}{c}
v_1+v_7=0,\qquad v_2+v_3+v_4=v_1,\qquad v_4+v_5+v_6=2v_1,
\\[3pt]v_5+v_6+v_7=v_2+v_3,\qquad v_1+v_2+v_3=v_5+v_6.\end{array}\]
One may then reconstruct the variety $Y$ out of these relation, and compute its nef, movable and effective cones.
The effective cone $\ol{\Eff(Y)}$ contains six Mori chambers, $N_1,\dots, N_6$. We have represented them in Figure \ref{fig:batyrevG11}.\par 

\begin{figure}[h!!]\label{fig:batyrevG11}
\begin{tikzpicture}[scale=4]
% Draw the triangle
\draw[thick] (-1,0) -- (1,0) -- (0,1.73) -- cycle;

% Internal lines
\draw[thin] (0,0) -- (0,1.73);
\draw[thin] (-0.25, 0.5) -- (1,0);
\draw[thin] (-0.25, 0.5) -- (0,1.73);
\draw[thin] (-0.25, 0.5) -- (-1,0);
\draw[thin] (-0.25, 0.5) -- (0,0);
% CANCEL next lines to unhighlight the nef cone
\draw[thick] (-0.25, 0.5) -- (0,0);
\draw[thick] (-0.25, 0.5) -- (0,0.4);
\draw[thick] (0, 0.4) -- (0,0);
%

% Labels for regions
\node at (-0.4,0.2) {\(N_2\)};
\node at (0.25,0.2) {\(N_3\)};
\node at (-0.35,0.8) {\(N_5\)};
\node at (-0.07,0.3) {\(N_1\)};
\node at (-0.1,0.8) {\(N_4\)};
\node at (0.25,0.8) {\(N_6\)};

% Label for Eff(Y)
\node at (0.9,0.9) {\(\overline{\mathrm{Eff}(Y)}\)};

% Vertices
\filldraw[black] (0,0) circle (0.02) node[below left] {};
\filldraw[black] (1,0) circle (0.02) node[below right] {};
\filldraw[black] (-1,0) circle (0.02) node[below right] {};
\filldraw[black] (0,1.73) circle (0.02) node[above] {};
\filldraw[black] (-0.25,0.5) circle (0.02) node[below right] {};
\filldraw[black] (0,0.4) circle (0.02) node[below right] {};
\end{tikzpicture}
\caption{The Mori chamber decomposition of the Fano $4$-fold~$Y$.\label{fig:batyrevG11}}
\end{figure}
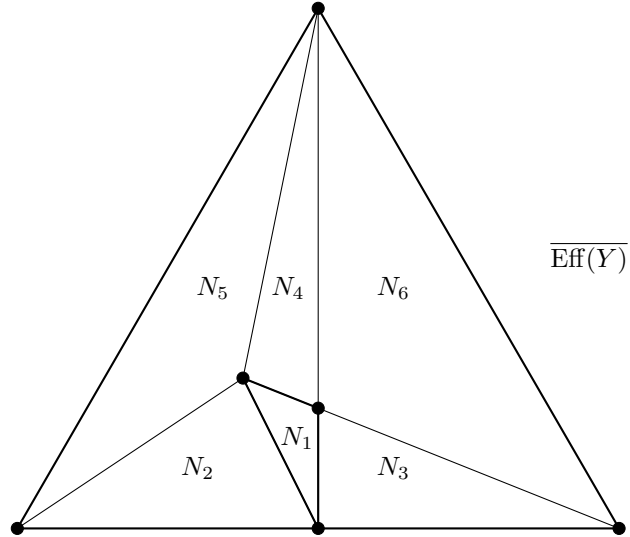

We have chosen the numbering of the chambers so that:
\begin{itemize}[leftmargin=\xx pt]
\item the closure of $N_1$ is the nef cone of $Y$;
\item the closures of $N_2,N_3$ correspond to the nef cones of two small $\Q$-factorial modifications of $Y$, that we denote $Y_2,Y_3$, respectively; in particular $\ol{\Mov(Y)}$ is equal to $\ol{N_1\cup N_2\cup N_3}$;
\item the Mori chambers $N_4,N_5,N_6$ correspond to divisorial contractions $Y_4,Y_5,Y_6$ of $Y,Y_2,Y_3$, respectively.
\end{itemize}
Moreover, one may check that $Y_4,Y_6$ are $\Q$-factorial, but not smooth, whereas $Y,Y_2,Y_3,Y_5$ are smooth.
The relations among these six varieties can be expressed in the following diagram, where dashed arrows correspond to flips and arrows labelled with a $d$ correspond to divisorial contractions:
\[
\xymatrix@C=15mm{Y_5&Y_4\ar@{-->}[r]\ar@{-->}[l]&Y_6\\Y_2\ar[u]_d&Y\ar@{-->}[r]\ar@{-->}[l]\ar[u]_d&Y_3\ar[u]_d}
\] 
Choosing two line bundles $A,B$ in different chambers, corresponding to two varieties $Y_-,Y_+\in\{Y,Y_2,\dots,Y_6\}$, and choosing a divisor $H=B-A$, we construct a geometric realization of the birational map between $Y_-$ and $Y_+$. Obviously different choices of the divisors in the chambers of $Y_\pm$ and of the integer $m$ provide different geometric realizations containing information about different factorizations of the birational map $\phi:Y_-\dashrightarrow Y_+$.

As an example, set $Y_-:=Y_5$, $Y_+:=Y_6$, and consider two different choices of divisors:
\[(A,B)=\left\{\begin{array}{l}(A_1,B_1)=(D_1+2D_6+D_7,D_1+2D_2+D_7)\\(A_2,B_2)=(D_1+6D_6+D_7,D_1+6D_2+D_7)
\end{array}\right.\]
where $D_1,\dots,D_7$ are the torus-invariant divisors corresponding to the vertices of the Fano polytope $v_1,\dots,v_7$, respectively.
We have represented their numerical classes in Figure \ref{fig:batyrevG12}.

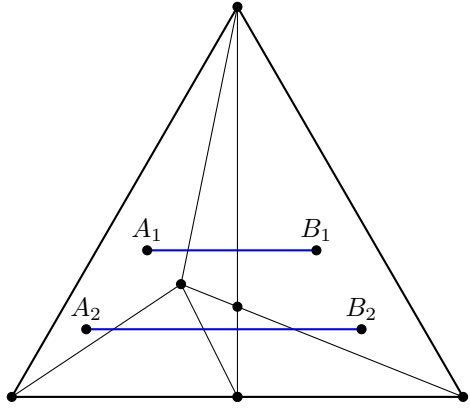
\begin{figure}[h!!]
\begin{tikzpicture}[scale=3]
% Draw the triangle
\draw[thick] (-1,0) -- (1,0) -- (0,1.73) -- cycle;

% Internal lines
\draw[thin] (0,0) -- (0,1.73);
\draw[thin] (-0.25, 0.5) -- (1,0);
\draw[thin] (-0.25, 0.5) -- (0,1.73);
\draw[thin] (-0.25, 0.5) -- (-1,0);
\draw[thin] (-0.25, 0.5) -- (0,0);
% CANCEL next lines to unhighlight the nef cone
%\draw[thick] (-0.25, 0.5) -- (0,0);
%\draw[thick] (-0.25, 0.5) -- (0,0.4);
%\draw[thick] (0, 0.4) -- (0,0);
%
\draw[thick, blue] (0.55,0.3) -- (-0.67,0.3);
\draw[thick, blue] (0.35,0.65) -- (-0.4,0.65);

% Label for Eff(Y)
%\node at (0.9,0.9) {\(\overline{\mathrm{Eff}(Y)}\)};

% Vertices
\filldraw[black] (0.55,0.3) circle (0.02) node[above] {$B_2$};
\filldraw[black] (-0.67,0.3) circle (0.02) node[above] {$A_2$};
\filldraw[black] (0.35,0.65) circle (0.02) node[above] {$B_1$};
\filldraw[black] (-0.4,0.65) circle (0.02) node[above] {$A_1$};
\filldraw[black] (0,0) circle (0.02) node[below left] {};
\filldraw[black] (0,0) circle (0.02) node[below left] {};
\filldraw[black] (1,0) circle (0.02) node[below right] {};
\filldraw[black] (-1,0) circle (0.02) node[below right] {};
\filldraw[black] (0,1.73) circle (0.02) node[above] {};
\filldraw[black] (-0.25,0.5) circle (0.02) node[below right] {};
\filldraw[black] (0,0.4) circle (0.02) node[below right] {};
\end{tikzpicture}
\caption{Two geometric realizations of the birational map between two birational contractions of the Fano $4$-fold $Y$.}\label{fig:batyrevG12}
\end{figure}

In the first case, we obtain a projective toric $5$-fold, singular but $\Q$-factorial, determined by a moment polytope with $18$ vertices in $\MM((\C^*)^5)_\R$, where
eight vertices correspond to the sink  $Y_5$ of the action, 
eight correspond to the source $Y_6$,  
and two vertices correspond to two isolated inner fixed points, 
and they are associated with two flips that we need to perform to transform $Y_5$ into $Y_6$.  The information about the factorization of $\phi$ that the action induces is shown as the output of the command {\tt action\_info()}: 
\begin{python}
G1=geometric_realization(Y,A1,B1)
action_info(G1)
\end{python}
\noindent\begin{minipage}{\textwidth}
\begin{lstlisting}[language=python,basicstyle=\ttfamily\footnotesize,breakindent=.5\textwidth, breaklines=true]
The criticality of the action is 3
The weights are [0,4,9,12]
The polytopes of fixed point components have [8,1,1 8] vertices
The map GX_0 --> GX_1 is a flip
The map GX_1 --> GX_2 is a flip
The variety is complete, is Q-factorial, is not smooth, is Fano
\end{lstlisting}
\end{minipage}\par\medskip

In other words, it decomposes $\phi$ as a composition of two flips, namely:
\[
\xymatrix@C=15mm{Y_5&Y_4\ar@{-->}[r]\ar@{-->}[l]&Y_6
}
\]

If we consider the second set of divisors $(A,B)$, we obtain a $\Q$-factorial toric $5$-fold, determined by a moment polytope with $22$ vertices in $\MM((\C^*)^5)_\R$, 
together with a $\C^*$-action with the following properties:
\par\medskip
\noindent\begin{minipage}{\textwidth}
\begin{lstlisting}[language=python,basicstyle=\ttfamily\footnotesize,breakindent=.5\textwidth,breaklines=true]
The criticality of the action is 5
The weights are [0,2,5,7,8,12]
The polytopes of fixed point components have [8,2,1,1,2,8] vertices
The map GX_0 --> GX_1 is a divisorial extraction
The map GX_1 --> GX_2 is a flip
The map GX_2 --> GX_3 is a flip
The map GX_3 --> GX_4 is a divisorial contraction
The variety is complete, is Q-factorial, is not smooth, is not Fano
\end{lstlisting}
\end{minipage}\par\medskip
In other words, the realization provides the following factorization of $\phi$:
\[
\xymatrix@C=15mm{Y_5&&Y_6\\Y_2\ar[u]_d&Y\ar@{-->}[r]\ar@{-->}[l]&Y_3\ar[u]_d
}
\] 

Finally, let us  illustrate the fact that we may construct geometric realizations that are Fano varieties but not $\P^1$-bundles. That is what happens, for instance, in the case in which we set:
\[
H:=D_1+D_3-D_5,
\]
(one may check that $-K_Y-H$, $-K_Y+H$ belong, respectively, to the Mori chambers $N_2$, $N_6$) and use the function {\tt \verb|Fano_realization(Y,H)|}:

\begin{python}
F=Fano_realization(Y,H)
action_info(F)
\end{python}
\noindent\begin{minipage}{\textwidth}
\begin{lstlisting}[language=python,basicstyle=\ttfamily\footnotesize,breakindent=.5\textwidth, breaklines=true]
The criticality of the action is 4
The weights are [-12,-1,7,9,12]
The polytopes of fixed point components have [12,1,2,1,8] vertices
The map GX_0 --> GX_1 is a flip
The map GX_1 --> GX_2 is a divisorial contraction
The map GX_2 --> GX_3 is a flip
The variety is complete, is Q-factorial, is not smooth, is Fano
\end{lstlisting}
\end{minipage}\par\medskip

%%%%%%%%%%%%%%%%%%%%%%%%%%%%%%%%%%%%%%%%%%%%%%%%%%%%%%%%%%%%%%%%%%%%%%%%%%%%%
%%%%%%%%%%%%%%%%%%%%%%%%%%%%%%%%%%%%%%%%%%%%%%%%%%%%%%%%%%%%%%%%%%%%%%%%%%%%%
%%
%% END OF SECTION 4
%%
%%%%%%%%%%%%%%%%%%%%%%%%%%%%%%%%%%%%%%%%%%%%%%%%%%%%%%%%%%%%%%%%%%%%%%%%%%%%%
%%%%%%%%%%%%%%%%%%%%%%%%%%%%%%%%%%%%%%%%%%%%%%%%%%%%%%%%%%%%%%%%%%%%%%%%%%%%%

%%%%%%%%%%%%%%%%%%%%%%%%%%%%%%%%%%%%%%%%%%%%%%%%%%%%%%%%%%%%%%%%%%%%%%%%%%%%%
%%%%%%%%%%%%%%%%%%%%%%%%%%%%%%%%%%%%%%%%%%%%%%%%%%%%%%%%%%%%%%%%%%%%%%%%%%%%%
%%
%% SECTION 5
%%
%%%%%%%%%%%%%%%%%%%%%%%%%%%%%%%%%%%%%%%%%%%%%%%%%%%%%%%%%%%%%%%%%%%%%%%%%%%%%
%%%%%%%%%%%%%%%%%%%%%%%%%%%%%%%%%%%%%%%%%%%%%%%%%%%%%%%%%%%%%%%%%%%%%%%%%%%%%
\section{Geometric realizations of birational maps between MDSs}\label{sec:MDSs}

In this section we will consider the problem of  effectively constructing geometric realizations of birational transformations between Mori Dream Spaces, by means of our previous construction of toric geometric realizations. We will start by recalling some background facts about Mori embeddings in toric varieties.

%%%%%%%%%%%%%%%%%%%%%%%
\subsection{Mori embeddings}\label{ssec:moriemb}
%%%%%%%%%%%%%%%%%%%%%%%

We refer the interested reader to \cite[Section~2]{HuKeel} for details. Let $Y$ be an MDS, $\rho$ be the Picard number of $Y$, and $\cO_Y(D_1),\ldots,\cO_Y(D_\rho)$ be a basis of $\Pic(Y)_\Q$. 
We define the \emph{Cox ring} of $Y$ (with respect to the choice of $D_1,\dots,D_\rho$) as the graded ring 
$$R:=R(Y;D_1,\dots,D_\rho)= \bigoplus_{(m_1,\dots,m_\rho)\in \Z^\rho} \HH^0\left(Y,\sum_{i=1}^\rho m_iD_i\right)$$
where the ring structure is induced by the multiplication of sections. The group $T:=\Hom(\Z^\rho,\C^*)$ is a $\rho$-dimensional torus, and  has a natural action on $R$ defined as follows: given $L=\sum_{i=1}^\rho m_iD_i$, $s\in \HH^0(X,L)$, and $t\in T$, we set $t\cdot s = t^Ls$. Moreover, recall (cf. \cite[Proposition 2.9]{HuKeel}) that the variety $Y$ is a GIT quotient of $\Spec(R)$ by the $T$-action. We will elaborate on this later in this section.

Since $Y$ is an MDS, $R$ is generated as a $\C$-algebra by a finite number of homogeneous elements $x_j\in \HH^0(Y,L_j)$, $L_j\in \bigoplus_{i=1}^\rho\Z D_i$,  
$j=1,\dots,N$. Hence there exists a surjective morphism
$$ \pi: \cdR := \C[x_1,\ldots,x_N]\to R,$$ 
which corresponds to an inclusion $\Spec(R)\hookrightarrow \Spec(\cdR)=\C^N$. 
By definition of $\cdR$, we may extend the action of the Picard torus on $R$ to a $T$-action on the polynomial ring $\cdR$, defined as  $t\cdot x_j = t^{L_j}x_j$. Moreover, since $\pi$ is $T$-equivariant, so is the inclusion $\Spec(R)\hookrightarrow \Spec(\cdR)=\C^N$.

Let $D\in \Pic(Y)$ be an ample line bundle on $Y$, and consider the linearization of the trivial bundle over $\C^N$ induced by $D$. By restriction, we get a linearization on the trivial bundle over $\Spec(R)$, so that $Y$ can be retrieved as the corresponding GIT-quotient of (the semistable set of) $\Spec(R)$:

\begin{lemma}\cite[Proposition~2.9]{HuKeel}\label{lem:MoriGIT}
	The GIT quotient $\Spec(R)\git^D T$ with respect to the linearization induced by $D$ is equal to $\Proj R(Y; D) \simeq Y$.
\end{lemma}

We may also use $D$ to linearize the trivial line bundle on $\C^N=\Spec(\cdR)$; the corresponding GIT quotient $\C^N\git^D T$ will be a toric variety, and we get an embedding (called {\em Mori embedding}):
$$
Y=\Spec(R)\git^D T\hooklongrightarrow \C^N\git^D T.
$$
More precisely, we may state that:

\begin{lemma}\cite[Proposition~2.11]{HuKeel}\label{lem:Moriemb}
	If $D$ is an ample divisor in $Y$ whose class is a general element in the ample cone of $Y$, then the GIT quotient $\cdY:=\C^N\git^D T$ is a quasi-smooth projective toric variety, and we have an embedding $Y\hookrightarrow \cdY$. Furthermore:
	\begin{enumerate}[leftmargin=\yy pt]
	\item\label{item:a} the restriction $\NU(\cdY)\to \NU(Y)$ is an isomorphism;
	\item\label{item:b} the isomorphism of (\ref{item:a}) sends $\ol{\Eff(\cdY)}$ bijectively to $\ol{\Eff(Y)}$;
	\item\label{item:c} every rational contraction of $Y$ extends to a toric rational contraction of $\cdY$. 
	\end{enumerate}
\end{lemma}

\begin{remark}\label{rem:Moriemb}
Note that a toric projective variety is quasi-smooth if its fan is simplicial; equivalently, it is $\Q$-factorial. 
In the above construction it is not true in general that $\Nef(\cdY)=\Nef(Y)$. In fact, the ample chamber of $Y$ in $\NU(Y)=\Mo(T)_\R$  
may contain a finite number of Mori chambers of $\cdY$. In other words, $\cdY$ depends on the particular choice of $D$, and it is only $\Q$-factorial for the general element of the ample cone.
\end{remark}

%%%%%%%%%%%%%%%%%%%%%%%
\subsection{Restricting toric geometric realizations}\label{ssec:restriction}
%%%%%%%%%%%%%%%%%%%%%%%

Now we will make use of Mori embeddings to construct effectively geometric realizations of birational transformations between Mori Dream Spaces, by embedding them appropriately into projective toric varieties. 

We will work again under the assumptions of Setup \ref{set:birational resolution}, considering a birational map $\phi:Y_-\dashrightarrow Y_+$ defined by two $\Q$-factorial birational contractions of a Mori Dream Space $Y$,
\[
\xymatrix{Y_-\ar@<1ex>@/^0.9pc/@{-->}[rrrr]^{\phi}&&Y\ar@{-->}[ll]_{\pi_-}\ar@{-->}[rr]^{\pi_+}&&Y_+}
\]
supported by big divisors $A,B\in \Pic(Y)$, respectively. 

As in Construction \ref{cons:geomreal} we choose $\ell$ so that $H:=(B-A)/\ell$ is Cartier, and consider a sufficiently large integer $m$, such that the following conditions hold:
\begin{enumerate}[leftmargin=\yy pt]
\item[(a)] $A$ and $B=A+\ell H$ are inner points of two Mori chambers of $\ol{\Eff(Y)}$;
\item[(b)] there exists $a\in[0,\ell m]\cap \Z$ such that the divisors $D:=mA+aH$ and $D+H$ are ample on $Y$.
\end{enumerate}
 As above, we will consider the Cox ring $R$ of $Y$ with respect to the choice of a basis $\{D_1,\dots,D_\rho\}$ of $\Pic(Y)_\Q$,   the corresponding torus $T=\Hom(\Z^\rho,\C^*)$, and a graded presentation $\pi:\cdR=\C[x_1,\dots,x_N]\to R$, so that we have a Mori embedding:
\[Y=\Spec(R)\git^DT\hooklongrightarrow \cdY:=\C^N\git^DT,\]
for some ample divisor $D\in \Pic(Y)$. Identifying $\NU(Y)$ with $\NU(\cdY)$ as in Lemma \ref{lem:Moriemb}(\ref{item:a}), we may assume, without loss of generality, that $D$ is an ample Cartier divisor in $\cdY$. We will further assume that:
\begin{enumerate}[leftmargin=\yy pt]
\item[(c)] the classes of $D,D+H$ belong to the same Mori chamber of $\cdY$, i.e., $\cdY$ is $\Q$-factorial and the toric variety $\C^N\git^{D+H}T$ is isomorphic to $\cdY$.
\end{enumerate}
Note that, up to substituting $Y$ with a birational contraction dominating $Y_\pm$, we may always find $A,B,\ell, m$ satisfying the conditions (a),(b),(c).

By Lemma \ref{lem:Moriemb} the $\Q$-factorial birational contractions $\pi_\pm$ of $Y$ extend to toric $\Q$-factorial birational contractions of the toric variety $\cdY$, that we denote by $\pi_\pm:\cdY\dashrightarrow\cdY_\pm$. In particular the map $\phi=\pi_+\circ\pi_-^{-1}$ extends to a birational map $\cdY_-\dashrightarrow \cdY_+$ that, abusing notation, we denote again by $\phi$:
\[
\xymatrix{\cdY_-
\ar@<1ex>@/^0.9pc/@{-->}[rrrr]^{\phi=\pi_+\circ\pi_-^{-1}}&&\cdY\ar@{-->}[ll]_{\pi_-}\ar@{-->}[rr]^{\pi_+}&&\cdY_+}
\]
Note that the maps $\pi_-,\pi_+$ are supported by big $\Q$-divisors whose restrictions to $Y$ are  respectively $A,B$. Abusing notation, we denote them by $A,B\in\Pic(\cdY)_\Q$ and, up to exchanging them with a multiple, we may assume that $A,B$, and $H=(B-A)/\ell$ are Cartier divisors in $\cdY$.  
We may now construct a toric geometric realization $\cdX$ of $\phi:\cdY_-\dashrightarrow \cdY_+$, using  the combinatorial construction described in Section \ref{section:toric}. As in Remark \ref{rem:geomrealunprun}, the assumptions (b), (c) imply that $\cdX$ is an unpruning of $\cdW=\P(\cO_\cdY\oplus\cO_\cdY(H))$, and we can write:
\[\begin{split}
\cdY=& \Proj \bigoplus_{k\geq 0} \HH^0(\cdY,\cO_{\cdY}(mkA+kaH)),\\
\cdW=&\Proj \bigoplus_{k\geq 0}\HH^0(\cdY,S^k(\cO_{\cdY}(mA+aH)\oplus \cO_{\cdY}(mA+(a+1)H)))=\\
   =& \Proj \bigoplus_{k\geq 0}\bigoplus_{b=ka}^{k(a+1)}\HH^0(\cdY,\cO_{\cdY}(mkA+bH)),\\
\cdX=&\Proj \bigoplus_{k\geq 0}\bigoplus_{b=0}^{mk}\HH^0(\cdY,\cO_{\cdY}(mkA+bH)).
\end{split}
\]
We denote the graded rings above as 
$S(\cdY)$,  $S(\cdW)$, $S(\cdX)$, respectively. The maps $\cdY\longleftarrow \cdW\dashleftarrow \cdX$ 
are then induced by the natural inclusions of graded rings:
\[
S(\cdY)\hooklongrightarrow S(\cdW)\hooklongrightarrow S(\cdX).
\]

Our construction is compatible with the restrictions to $Y$, so the geometric realization $X$ of the birational map $\phi:Y_-\dashrightarrow Y_+$, constructed upon divisors $A$ and $B$, will be a subvariety $X\subset\cdX$; more precisely, it will be the strict transform of $W:=\P(\cO_Y\oplus\cO_Y(H))\subset \cdW$ into $\cdX$, so we will have  the following diagram  of Cartesian squares:
\[
\xymatrix@C=11.5mm{\cdY&\cdW\ar[l]\ar@{-->}[r]&\cdX\\
Y\ar@{^(->}+<0pt,10pt>;[u]&W\ar[l]\ar@{^(->}+<0pt,10pt>;[u]\ar@{-->}[r]&X\ar@{^(->}+<0pt,10pt>;[u]}
\]
In particular, if $I(Y)\subset S(\cdY)$ is the ideal of $Y$ in $\cdY$, it follows that the ideal of $X$ in $\cdX$ will be the extension:
\[
I(Y)\otimes S(\cdX).
\]
In other words, $X$ is given in $\cdX$ by the same homogeneous polynomial equations of $Y$ in $\cdY$.
This fact takes a particularly simple form in the case in which $Y$ is a hypersurface in $\cdY$, so that it can be seen as the set of zeroes of a section of a  line bundle. In the next section we will present an example of this kind.

%%%%%%%%%%%%%%%%
\subsection{An example}\label{ssec:example}
%%%%%%%%%%%%%%%%

The following example aims to illustrate the results of the previous section; namely we will construct a geometric realization of a birational transformation between MDSs by means of the Mori embeddings into toric varieties. For the sake of clarity, we will use the same notation as in the previous section.\par\medskip

\noindent\textbf{The toric varieties $\cdY, \cdY_{\pm}$.} Let $\cdY_-=\cdY$ be the (Grothendieck) projectivization of the bundle $\cO_{\P^2}(0,1,1):=\cO_{\P^2}\oplus \cO_{\P^2}(1)\oplus \cO_{\P^2}(1)$.  It is a smooth toric Fano $4$-fold, corresponding to number $9$ of Batyrev's list (cf. \cite[Proposition 3.1.1]{batyrev}). We may construct $\cdY_-$ as follows:

\begin{python}
P2 = toric_varieties.P2()
M = matrix([[0,0,0],[0,0,1],[0,0,1]])
Y = projective_bundle(P2,M)
\end{python}

The rays of the fan defining $\cdY_-$ are generated by the vectors:
\[\begin{array}{l}\rho_1=(-1,-1,1,1),\quad \rho_2=(0,0,-1,-1),\quad \rho_3=(0,0,0,1),\\\rho_4=(0,0,1,0),\quad \rho_5=(0,1,0,0),\quad \rho_6=(1,0,0,0),\end{array}\]
corresponding to irreducible $(\C^*)^4$-invariant divisors $D_{1},\dots, D_{6}$, respectively. The Picard group of $\cdY_-$ is generated by the linear equivalence classes of $D_1,D_2$, and we have the relations $[D_1]=[D_5]=[D_6]$, $[D_3]=[D_4]=[D_2-D_1]$. The nef cone of $\cdY_-$ is generated by the numerical classes $[D_1],[D_2] \in \NU(\cdY_-)$; the movable cone is equal to the effective cone and is generated by the classes $[D_1],[D_3]$. 

The section corresponding to the projection $\cO_{\P^2}(0,1,1)\to \cO_{\P^2}$ can be flipped, obtaining an SQM $\phi: \cdY_-\dashrightarrow \cdY_+$, with $\cdY_+=\P(\cO_{\P^1}(0,1,1,1))$.  The nef cone of $\cdY_+$ is generated by the numerical classes of $[D_2],[D_3]$, and the natural projection of $\cdY_+$ to $\P^1$ is given by the extremal ray generated by $[D_3]$.\par\medskip

\noindent\textbf{The toric geometric realization $\cdX$.} Let us construct a toric geometric realization $\cdX$ of the SQM $\phi: \cdY_- \dashrightarrow \cdY_+$. To this end, let $A = 2D_1+2D_2\in \Amp(\cdY_-)$, $B=-D_1+2D_2\in \Amp(\cdY_+)$ and $\ell=3$.

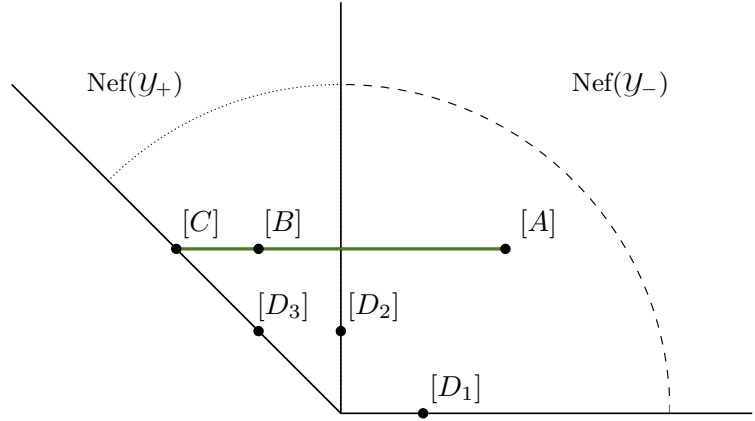
\begin{figure}[!ht]
	\centering
	\resizebox{0.8\textwidth}{!}{
		\begin{tikzpicture}
			\draw [ line width=0.2mm](0,0) to (5,0);
			\draw [ line width=0.2mm](0,0) to (0,5);
			\draw [ line width=0.2mm](0,0) to (-4,4);
			\node at (1,0) [circ] {}; %D1
				\node at (1.4,0.3) {$[D_1]$};
			\node at (0,1) [circ] {}; %D5
				\node at (0.4,1.3) {$[D_2]$};
			\node at (-1,1) [circ] {}; %D3
				\node at (-0.7,1.3) {$[D_3]$};
			\draw [color={rgb,255:red,79; green,122; blue,40}, line width=0.4mm](-2,2) to[short] (2,2);   
			\node at (-2,2) [circ] {}; %C
				\node at (-1.7,2.3) {$[C]$};
			\node at (-1,2) [circ] {}; %B
				\node at (-0.7,2.3) {$[B]$};
			\node at (2,2) [circ] {}; %A
				\node at (2.4,2.3) {$[A]$};
			
			\node at (-2.5,4) {\small$\text{Nef}(\cdY_+)$};
			\node at (3.4,4) {\small$\text{Nef}(\cdY_-)$};
			
			\begin{scope}
				\clip (0,0) rectangle (4,4);
				\draw[dashed] (0,0) circle(4);
			\end{scope}

			\draw[densely dotted] (0,0) -- (0,4) arc [start angle=90, delta angle=45, radius=4cm] -- (0,0);
		\end{tikzpicture}
	}
	\caption{The effective cone of $\cdY_-$ and its Mori chamber decomposition.\label{fig:Conestoric1}}
\end{figure}

\begin{python}
A = [2,2,0,0,0,0]
B = [-1,2,0,0,0,0]
X = geometric_realization(Y,A,B,3)
\end{python}
\noindent\begin{minipage}{\textwidth}
\begin{lstlisting}[language=python,basicstyle=\ttfamily\footnotesize,breakindent=.5\textwidth,breaklines=true]
The criticality of the action is 2
The weights are [0, 2, 3]
The polytopes of fixed point components have [9, 1, 8] vertices
The map GX_0 --> GX_1 is a flip
The variety is complete, is Q-factorial, is smooth, is not Fano
\end{lstlisting}
\end{minipage}\par\medskip

Now we choose instead of $B$ the divisor $C = -2D_1+2D_2 = 2D_3$, and $\ell=4$. Note that $C$ belongs to an extremal ray of $\Eff(\cdY_\pm)$, so we are not in the situation of Set-up \ref{set:birational resolution}. Nevertheless, the geometric realization construction of Section \ref{section:toric} still works in this case, and it provides a geometric realization $\cdX_2$ whose source is $\P^1$, and such that the pruning map $\beta: \cdX\to \cdX_2$ is the blowup of $\cdX_2$ along $\P^1$. 
\par\medskip

\begin{python}
C = [-2,2,0,0,0,0]
X2 = geometric_realization(Y,A,C,4)
\end{python}
\noindent\begin{minipage}{\textwidth}
	\begin{lstlisting}[language=python,basicstyle=\ttfamily\footnotesize,breakindent=.5\textwidth,breaklines=true]
The criticality of the action is 2
The weights are [0, 2, 4]
The polytopes of fixed point components have [9, 1, 2] vertices
The map GX_0 --> GX_1 is a flip
The variety is complete, is Q-factorial, is smooth, is Fano
\end{lstlisting}
\end{minipage}

Moreover, by analyzing its combinatorial data, one may show that $\cdX_2 \simeq \P(\cO_{\P^3}(0,1,1))$.  The induced $\C^*$-action in $\P(\cO_{\P^3}(0,1,1))$ descends via the natural projection $\pi:\P(\cO_{\P^3}(0,1,1))\to \P^3$ to an equalized action with sink a plane $\P^2$ and source a point $P\in\P^3$. Then the sink of $\P(\cO_{\P^3}(0,1,1))$ is $\pi^{-1}(\P^2)$ and its source is a line in the fiber $\pi^{-1}(P)$, not meeting the minimal section of $\pi$. It has a unique inner fixed point, contained in the fiber $\pi^{-1}(P)$. 
The Picard group of $\cdX$ is then generated by $\cO_{\cdX}(\cdY_+)$,  $\beta^*\cO_{\cdX_2}(1)$ and $\beta^*\pi^*\cO_{\P^3}(1)$.\par\medskip

\noindent\textbf{The anticanonical sections.} Let us now consider a general anticanonical section $Y_-$ of the toric Fano $4$-fold $\cdY_-= \P(\cO_{\P^2}(0,1,1))$. Note that the anticanonical line bundle of $\cdY_-$ is $\cO_{\cdY_-}(3)\otimes p^*\cO_{\P^2}(1)$; then $Y_-$ meets the minimal section of $\P(\cO_{\P^2}(0,1,1))$ along a line, whose normal bundle in $Y_-$ is $\cO_{\P^1}(-1)^{\oplus 2}$. The flip $\phi: \cdY_-\dashrightarrow \cdY_+$ restricts to the flip of $\phi: Y_-\dashrightarrow Y_+$ along that line. At the level of geometric realizations, we consider $Y_-$ as a divisor in the sink $\cdY_-$ of $\cdX$, and extend this divisor to a divisor in $\cdX$ by means of the $\C^*$-action, in order to obtain a geometric realization $X$ of the map $\phi:Y_-\dashrightarrow Y_+$. The variety $X$ is a $\C^*$-invariant divisor in $\cdX$, with associated line bundle $\cO_{\cdX}(-4\cdY_+)\otimes b^*\cO_{\cdX_2}(3)\otimes b^*p^*\cO_{\P^3}(1)=\omega_{\cdX}^{-1}\otimes (\cO_{\cdX}(\cdY_+)\otimes b^*p^*\cO_{\P^3}(-1))$.

%%%%%%%%%%%%%%%%%%%%%%%%%%%%%%%%%%%%%%%%%%%%%%%%%%%%%%%%%%%%%%%%%%%%%%%%%%%%%
%%%%%%%%%%%%%%%%%%%%%%%%%%%%%%%%%%%%%%%%%%%%%%%%%%%%%%%%%%%%%%%%%%%%%%%%%%%%%
%%
%% END OF SECTION 5
%%
%%%%%%%%%%%%%%%%%%%%%%%%%%%%%%%%%%%%%%%%%%%%%%%%%%%%%%%%%%%%%%%%%%%%%%%%%%%%%
%%%%%%%%%%%%%%%%%%%%%%%%%%%%%%%%%%%%%%%%%%%%%%%%%%%%%%%%%%%%%%%%%%%%%%%%%%%%%

\bibliographystyle{plain}
\bibliography{biblioPostdoc}
\end{document}

%% file: definitions.tex
\usepackage{enumitem}
\newcommand\ignore[1]{}

\DeclareMathOperator{\HH}{H}

\def\C{{\mathbb C}}

\def\P{{\mathbb P}}
\def\Q{{\mathbb Q}}
\def\R{{\mathbb R}}
\def\Z{{\mathbb Z}}

\DeclareMathAlphabet{\mathdutchcal}{U}{dutchcal}{m}{n}
\SetMathAlphabet{\mathdutchcal}{bold}{U}{dutchcal}{b}{n}
\DeclareMathAlphabet{\mathdutchbcal}{U}{dutchcal}{b}{n}

\def\cA{{\mathcal A}}

\def\cC{{\mathcal C}}

\def\cO{{\mathcal{O}}}

\def\cdR{{\mathdutchcal R}}

\def\cdW{{\mathdutchcal W}}

\def\cY{{\mathcal Y}}
\def\cdX{{\mathdutchcal X}}
\def\cdY{{\mathdutchcal Y}}
\def\Q{{\mathbb{Q}}}

\def\operatorname#1{\mathop{\rm #1}\nolimits}

\def\Proj{\operatorname{Proj}}

\def\Hom{\operatorname{Hom}}

\def\Pic{\operatorname{Pic}}
\def\Hom{\operatorname{Hom}}

\def\Spec{\operatorname{Spec}}

\def\Nef{{\operatorname{Nef}}}

\def\NU{{\operatorname{N^1}}}

\def\Eff{{\operatorname{Eff}}}
\def\Mov{{\operatorname{Mov}}}

\def\Amp{\operatorname{Amp}}

\def\GX{\mathcal{G}\!X}

\newcommand{\pb}{\ar@{}[dr]|{\text{\pigpenfont J}}}
\def\ol{\overline}

\newcommand{\xdasharrow}[2][->]{
\tikz[baseline=-\the\dimexpr\fontdimen22\textfont2\relax]{
\node[anchor=south,font=\scriptsize, inner ysep=1.5pt,outer xsep=2.2pt](x){#2};
\draw[shorten <=3.4pt,shorten >=3.4pt,dashed,#1](x.south west)--(x.south east);
}}
\newcommand{\hooklongrightarrow}{\lhook\joinrel\longrightarrow}

\def\Mo{\operatorname{\hspace{0cm}M}}

\def\MM{\mathrm{M}}
\def\NN{\mathrm{N}}

\newcommand{\git}{\mathbin{/\mkern-6mu/}}